\documentclass[11pt,a4paper]{amsart}
\usepackage{amsfonts}
\usepackage{amsthm}
\usepackage{amsmath}
\usepackage{amscd}
\usepackage{amssymb}
\usepackage{t1enc}
\usepackage[mathscr]{eucal}
\usepackage{indentfirst}
\usepackage{graphicx}
\usepackage{graphics}
\usepackage{pict2e}
\usepackage{epic}
\usepackage{tikz-cd}

\usepackage[style = alphabetic]{biblatex}
\renewbibmacro{in:}{}
\numberwithin{equation}{section}
\usepackage[margin=2.9cm]{geometry}
\usepackage{epstopdf}
\usepackage{comment}
\usepackage[colorinlistoftodos]{todonotes}
\usepackage[toc,page]{appendix}
\usepackage{mathtools}
\usepackage{physics}

 \definecolor{winered}{rgb}{0.5,0,0}
 \usepackage{hyperref}
 \hypersetup{citecolor = winered, linkcolor = winered, menucolor = winered, colorlinks = true}

\newtheorem{Thm}{Theorem}[section]

\newtheorem{Cor}[Thm]{Corollary}
\newtheorem{Prop}[Thm]{Proposition}

\theoremstyle{definition}
\newtheorem{Def}[Thm]{Definition}
\newtheorem{Rem}[Thm]{Remark}
\newtheorem{Ex}[Thm]{Example}
\newtheorem{Cons}[Thm]{Construction}
\newtheorem{Prob}{Problem}

\newcommand{\Z}{\mathbb{Z}}

\newcommand{\conj}{\overline}

\newcommand{\cat}{\mathsf}

\newcommand{\colim}{\text{colim }}

\RequirePackage{tikz-cd}
\RequirePackage{amssymb}
\usetikzlibrary{calc}
\usetikzlibrary{decorations.pathmorphing}

\tikzset{curve/.style={settings={#1},to path={(\tikztostart)
    .. controls ($(\tikztostart)!\pv{pos}!(\tikztotarget)!\pv{height}!270:(\tikztotarget)$)
    and ($(\tikztostart)!1-\pv{pos}!(\tikztotarget)!\pv{height}!270:(\tikztotarget)$)
    .. (\tikztotarget)\tikztonodes}},
    settings/.code={\tikzset{quiver/.cd,#1}
        \def\pv##1{\pgfkeysvalueof{/tikz/quiver/##1}}},
    quiver/.cd,pos/.initial=0.35,height/.initial=0}

\tikzset{tail reversed/.code={\pgfsetarrowsstart{tikzcd to}}}
\tikzset{2tail/.code={\pgfsetarrowsstart{Implies[reversed]}}}
\tikzset{2tail reversed/.code={\pgfsetarrowsstart{Implies}}}
\tikzset{no body/.style={/tikz/dash pattern=on 0 off 1mm}}

\title{Categorical models for path spaces}

\author[E. Minichiello]{Emilio Minichiello}
\address{E.M., Department of Mathematics, CUNY Graduate Center}
\email{\href{mailto:eminichiello@gradcenter.cuny.edu}{eminichiello@gradcenter.cuny.edu}}

\author[M. Rivera]{Manuel Rivera}
\address{M.R., Department of Mathematics, Purdue University}
\email{\href{mailto:manuelr@purdue.edu}{manuelr@purdue.edu}}

\author[M. Zeinalian]{Mahmoud Zeinalian}
\address{M.Z., Department of Mathematics, Lehman College of CUNY}
\email{\href{mailto:mahmoud.zeinalian@lehman.cuny.edu}{mahmoud.zeinalian@lehman.cuny.edu}}

\begin{document}

\maketitle

\begin{abstract}
    
%
We establish an explicit comparison between two constructions in homotopy theory: the left adjoint of the homotopy coherent nerve functor, also known as the rigidification functor, and the Kan loop groupoid functor. This is achieved by considering localizations of the rigidification functor, unraveling a construction of Hinich, and using a sequence of operators originally introduced by Szczarba in 1961. As a result, we obtain several combinatorial models for the path category of a simplicial set. We then pass to the chain-level and describe a model for the path category, now considered as a category enriched over differential graded (dg) coalgebras, in terms of a suitable algebraic chain model for the underlying simplicial set. This is achieved through a version of the cobar functor inspired by Lazarev and Holstein's categorical Koszul duality. As a consequence, we obtain a conceptual explanation of a result of Franz stating that there is a natural dg bialgebra quasi-isomorphism from the extended cobar construction on the chains of a reduced simplicial set to the chains on its Kan loop group.
\end{abstract}

\tableofcontents

\section{Introduction}

Any topological space $Y$ may be regarded as a category $\mathcal{P}(Y)$ enriched over topological spaces whose morphisms are invertible up to homotopy. The objects of $\mathcal{P}(Y)$ are the points of $Y$ and given any two points $a,b\in Y$ the corresponding mapping space is the set
 $$\mathcal{P}(Y)(a,b)=\{ \gamma: [0,r] \to X : r \in \mathbb{R}_{\geq 0}, \gamma \text{ is continuous, and }\gamma(0)=a, \gamma (r)=b \}$$ equipped with the compact-open topology. Composition is defined by concatenation of paths and the identity morphisms are given by constant paths with $r=0$. This gives rise to a functor $$\mathcal{P}: \cat{Top} \to \cat{Cat}_{\cat{Top}}$$ from the category of topological spaces to the category of topological categories (categories enriched over the monoidal category of topological spaces) called the \textit{path category} functor. For any point $b \in Y$, the topological monoid of endomorphisms $\mathcal{P}(Y)(b,b)$ is Moore's model for the based loop space of $Y$ at $b$. We also have a simplicial version of $\mathcal{P}$ defined as the composition 
 $$\mathbb{P}: \cat{sSet} \xrightarrow{|\cdot|} \cat{Top} \xrightarrow{\mathcal{P}} \cat{Cat}_{\cat{Top}} \xrightarrow{\text{Sing}} \cat{Cat}_{\cat{sSet}}.$$
Here $|\cdot|: \cat{sSet}\to \cat{Top}$ denotes the geometric realization functor from the category of simplicial sets to the category of topological spaces and $\text{Sing}: \cat{Cat}_{\cat{Top}} \to \cat{Cat}_{\cat{sSet}}$ the functor from topological categories to simplicial categories (categories enriched over the monoidal category of simplicial sets) that applies the singular complex functor at the level of mapping spaces. 

The functor $\mathbb{P}$ sends weak homotopy equivalences of simplicial sets to weak equivalences of simplicial categories. Recall that a map of simplicial categories is called a weak equivalence if it induces an equivalence on homotopy categories and a weak homotopy equivalence on all simplicial sets of morphisms. Furthermore, any simplicial set  $X$ may be recovered, up to weak homotopy equivalence, by applying a homotopy coherent version of the nerve functor to $\mathbb{P}(X)$.

In the present article, we study and compare different models for the functor $\mathbb{P}$. We introduce some notation before stating our main results. Denote by $F: \cat{Quiv} \to \cat{Cat}$ the functor from the category of quivers to the category of categories that sends a quiver (also known as a reflexive directed graph) to the category freely generated by it.  We also have a functor $j: \cat{Quiv} \to \cat{sSet},$ that sends a quiver to its associated $1$-skeletal simplicial set. The right adjoint of $j$ is the functor
$\mathcal{Q}: \cat{sSet} \to \cat{Quiv}$ that sends a simplicial set to the quiver determined by its $1$-skeleton. Denote by $i: \cat{Gpd} \to \cat{Cat}$ the fully faithful embedding from groupoids to categories and by $L:\cat{Cat} \to \cat{Gpd}$ its left adjoint, called the localization functor. We are interested in solutions to the following.

\begin{Prob} Construct a functor  $\mathbb{F}: \cat{sSet} \to \cat{Cat}_{\cat{sSet}}$ that fits into a commutative diagram
\begin{equation} \label{1-loc}
    \begin{tikzcd}
{\cat{Quiv}} \arrow[d, "j"] \arrow[r, "i \circ L \circ F"]  & {\cat{Cat}} \arrow[d, "\tau "] \\
\cat{sSet} \arrow[r, "\mathbb{F}"]  & {\cat{Cat}_{\cat{sSet}}}
\end{tikzcd}
\end{equation}
and such that, for any simplicial set $X$, the simplicial categories $\mathbb{F}(X)$ and $\mathbb{P}(X)$ are naturally weakly equivalent. The vertical right functor $\tau$ sends an ordinary category to the simplicial category obtained by considering sets of morphisms as discrete simplicial sets. The commutativity of diagram \ref{1-loc} says that $\mathbb{F}$ inverts (strictly) every $1$-simplex.
\end{Prob}
Our motivation for finding solutions to the above problem is to obtain small and tractable combinatorial and algebraic models for the path category that are useful in geometric contexts. The main idea is to use ordered sequences of simplices to construct higher dimensional mapping spaces while localizing the smallest set of morphisms necessary to obtain the correct homotopy type. This may be achieved through different combinatorial constructions.

In \cite{dwyer_kan_1984}, Dwyer and Kan studied a classifying space functor $$\overline{W}: \cat{Cat}_{\cat{sSet}} \to \cat{sSet}.$$ Denote the left adjoint of $\overline{W}$ by $$G: \cat{sSet} \to \cat{Cat}_{\cat{sSet}}.$$ Dwyer and Kan were mainly interested in the composition $$ G^{Kan}=\mathcal{L} \circ G:\cat{sSet}\to \cat{Gpd}_{\cat{sSet}}$$ where $\mathcal{L}: \cat{Cat}_{\cat{sSet}} \to \cat{Gpd}_{\cat{sSet}}$ is the functor from simplicial categories to simplicial groupoids obtained by applying the ordinary localization functor $L$ at each simplicial degree. The restriction of $G^{Kan}$ to the category of $0$-reduced simplicial sets is known in the literature as the \textit{Kan loop group functor}. Dwyer and Kan show that the functor $G^{Kan}$ together with its right adjoint $\overline{W}$ relate two models for the homotopy theory of homotopy types. We will denote by $$\mathbb{G}: \cat{sSet} \to \cat{Cat}_{\cat{sSet}}$$ the composition of functors $\mathbb{G}= \iota \circ G^{Kan}$, where $\iota: \cat{Gpd}_{\cat{sSet}} \to \cat{Cat_{sSet}}$ is the inclusion functor.

Another related, but combinatorially different construction is the functor $$\mathfrak{C}: \cat{sSet} \to \cat{Cat}_{\cat{sSet}},$$ called the \textit{rigidification functor}, which is left adjoint to Cordier's \textit{homotopy coherent nerve functor} $$ \mathfrak{N}: \cat{Cat}_{\cat{sSet}} \to \cat{sSet}.$$ The rigidification functor has been studied in detail in \cite{lurie2006higher} and \cite{DuSp11}, where the adjunction $(\mathfrak{C}, \mathfrak{N})$ is used to relate two models for the homotopy theory of $\infty$-categories. In the present article, we consider the following two localizations of $\mathfrak{C}$.
\begin{enumerate}
\item  Localize all $0$-morphisms to obtain a functor $$\widehat{\mathfrak{C}}: \cat{sSet} \to \cat{Cat}_{\cat{sSet}}$$ given by $$X \mapsto \mathfrak{C}(X)[\mathfrak{C}(X^1)^{-1}],$$ where $X^1$ denotes the $1$-skeleton of $X$.
\item Localize all morphisms to obtain a functor $$\mathbb{C}: \cat{sSet} \to \cat{Cat}_{\cat{sSet}}$$
given by
$$ X\mapsto \iota(\mathcal{L}( \mathfrak{C}(X)))= \iota( \mathfrak{C}(X)[\mathfrak{C}(X)^{-1}]).$$
\end{enumerate}
It follows from the fact that $\mathfrak{C}(X)$ is a cofibrant simplicial category that the natural map $$\mu_X: \widehat{\mathfrak{C}}(X) \to \mathbb{C}(X)$$ is a weak equivalence of simplicial categories 
Our first result provides an explicit comparison between $\mathfrak{C}$ and $G^{Kan}$. 

\begin{Thm}\label{theorem1} There is a natural transformation of functors $Sz: \mathfrak{C} \xRightarrow{ } G$ inducing a weak equivalence of simplicial groupoids
$$\mathcal{L}Sz_X: \mathcal{L}\mathfrak{C}(X) \xrightarrow{\simeq} G^{Kan}(X)$$ for any simplicial set $X$. Furthermore, the three functors $\widehat{\mathfrak{C}}$, $\mathbb{C}$, and $\mathbb{G}$ are all solutions to Problem 1.
\end{Thm}

The notation ``$Sz$'' stands for \textit{Szczarba}, since the natural transformation $Sz: \mathfrak{C} \xRightarrow{ } \mathbb{G}$ is constructed using a sequence of simplicial operators reminiscent of a construction of Szczarba while studying twisted cartesian products. The description of the natural transformation $Sz$ in terms of Szczarba's simplicial operators gives explicit formulas for a construction originally proposed by Hinich in  \cite{hinich2007homotopy}. In particular, one may deduce from the above theorem that the geometric realization of the Kan loop group of a $0$-reduced simplicial set $X$ is weakly equivalent to the based loop space of $|X|$ \textit{as a topological monoid.} A different proof of this fact may be found in \cite{berger1995}.


We now turn to the problem of constructing and comparing \textit{algebraic} models for the path category. 
Let $R$ be a fixed commutative ring and write $\otimes_R=\otimes$. The dg $R$-module of normalized simplicial chains equipped with Alexander-Whitney diagonal approximation gives rise to a lax monoidal functor $$C^{\Delta}_*: (\cat{sSet}, \times) \to (\cat{dgCoalg}_{R}^{\geq 0}, \otimes)$$ from simplicial sets to differential non-negatively graded counital coassociative coalgebras. There is an induced functor $$\mathcal{C}^{\Delta}_*: \cat{Cat}_{\cat{sSet}} \to \cat{Cat}_{ \cat{dgCoalg}_{R}^{\geq 0}},$$ obtained by applying $C^{\Delta}_*$ on mapping spaces, where $\cat{Cat}_{\cat{dgCoalg}_{R}^{\geq 0}}$ is the category of categories enriched over $(\cat{dgCoalg}_{R}^{\geq 0},\otimes)$. A map $f: \cat{C} \to \cat{D}$ in $\cat{Cat}_{\cat{dgCoalg}_{R}^{\geq 0}}$ is called a \textit{quasi-equivalence}, if it induces a quasi-isomorphism between the underlying dg $R$-modules of morphisms and an essentially surjective map between homotopy categories. 

Following the framework for categorical Koszul duality proposed in \cite{holstein-lazarev}, we define the category $\cat{cCoalg}_R$ of \textit{categorical $R$-coalgebras}. An object of $\cat{cCoalg}_R$ consists of the data $$(C,\Delta, \partial, h)$$ where $(C, \Delta)$ is a non-negatively graded coassociative counital $R$-coalgebra such that $C_0=R[P_C]$ is freely generated by a set of  ``objects'' $P_C$,  $\partial: C \to C$ is a coderivation of degree $-1$, and $h: C_2 \to R$ a ``curvature'' map measuring the failure of $\partial$ in squaring to zero such that certain properties and compatibilities are satisfied. The category $\cat{cCoalg}_R$ provides a convenient setting to define a ``many-object'' version of the cobar functor
$$\Omega: \cat{cCoalg}_R \to \cat{dgCat}^{\geq 0}_R,$$ where $\cat{dgCat}^{\geq 0}_R$ is the category of differential non-negatively graded categories over $R$. Categorical coalgebras are particular examples of \textit{pointed curved coalgebras} as defined in \cite{holstein-lazarev}.

Any simplicial set $X$ gives rise to a categorical coalgebra through a modified version of the normalized chains functor $$\widetilde{C}^{\Delta}_*: \cat{sSet} \to \cat{cCoalg}_R.$$ One of our observations is that, for any simplicial set $X$, $\widetilde{C}^{\Delta}_*(X)$ may be equipped with additional algebraic structure that we call a  $B_{\infty}$-\textit{coalgebra} structure. A $B_{\infty}$-coalgebra is a categorical coalgebra $C \in \cat{cCoalg}_R$ equipped with coassociative coproducts
$$\nabla_{x,y}: \Omega(C)(x,y)\to \Omega(C)(x,y) \otimes \Omega(C)(x,y)$$
for all objects $x$ and $y$ in $\Omega(C)$ making $\Omega(C)$ into a category enriched over the monoidal category of dg $R$-coalgebras. $B_{\infty}$-coalgebras form a category $\cat{B_{\infty}Coalg}_R$. Thus, the cobar construction can be regarded as a functor $$\mathbf{\Omega}: \cat{B_{\infty}Coalg}_R \to \cat{Cat}_{\cat{dgCoalg}_{R}^{\geq 0}}.$$


The following statement, which is our second result, describes how to obtain an algebraic model for the chains on the path category directly from the $B_{\infty}$-coalgebra of normalized chains on $X$. 

\begin{Thm}\label{theorem2} There exists a functor $\widetilde{\mathbf{C}}_*: \cat{sSet} \to \cat{B_{\infty}Coalg}$ satisfying the following properties.
\begin{enumerate}
    \item $\widetilde{\mathbf{C}}_*$ is a lift of $\widetilde{C}^{\Delta}_*$, i.e. there is a natural isomorphism of functors $\mathcal{U} \circ \widetilde{\mathbf{C}}_* \cong \widetilde{C}^{\Delta}_*$, where  $\mathcal{U}: \cat{B_{\infty}Coalg}_R \to \cat{cCoalg}_R$ denotes the forgetful functor. 
    \item There is a natural isomorphism of functors $F\circ \mathcal{Q} \cong \mathcal{S} \circ \mathbf{\Omega} \circ \widetilde{\mathbf{C}}_*$, where $$\mathcal{S}:\cat{Cat}_{\cat{dgCoalg}_{R}^{\geq 0}} \to \cat{Cat}$$ applies the ``set-like elements'' functor at the level of morphisms. 
    \item For any simplicial set $X$, there are natural quasi-equivalences
    $$\widehat{\mathbf{\Omega}}(\widetilde{\mathbf{C}}_*(X)) \xrightarrow{\mathbb{T}_X} \mathcal{C}^{\Delta}_*(\widehat{\mathfrak{C}}(X)) \xrightarrow{\mathcal{C}^{\Delta}_*(\iota (\mathcal{L} (Sz_X)) \circ \mu_X)}\mathcal{C}^{\Delta}_*(\mathbb{G}(X)),$$ where
    $$\widehat{\mathbf{\Omega}}:  \cat{B_{\infty}Coalg} \to \cat{Cat}_{\cat{dgCoalg}_{R}^{\geq 0}}$$ is defined on any $C \in \cat{B_{\infty}Coalg}$ by formally inverting all set-like elements in each dg coalgebra $\mathbf{\Omega}(C)(x,y)$ for all objects $x,y \in \mathbf{\Omega}(C)$.
\end{enumerate}
\end{Thm}

The proof of Theorem \ref{theorem2} uses a cubical interpretation of the dg category $\Omega(\widetilde{C}^{\Delta}_*(X))$. 
The map $\mathbb{T}_X$ is then induced by a triangulation map that sends an $n$-dimensional cube to a formal sum of $n$-dimensional simplices indexed by the $n!$ permutations in the symmetric group of $n$ elements. 

Taking into account higher structure on the chains of a simiplicial set allows us to achieve two goals: 1) extract functorially a set of morphisms to formally invert them, and 2) model the chains on the path category not only as a dg category but as a category enriched over dg coalgebras. 
This results in an algebraic model for the path category of an arbitrary simplicial set. More precisely, Theorems \ref{theorem1} and \ref{theorem2} together imply the following. 

\begin{Thm}\label{algmodelforpathcat} For any simplicial set $X \in \cat{sSet}$, the dg coalgebra enriched categories $\widehat{\mathbf{\Omega}}(\widetilde{\mathbf{C}}_*(X))$ and $\mathcal{C}_*^{\Delta}(\mathbb{P}(X))$ are naturally quasi-equivalent.
\end{Thm}

Our argument also provides a conceptual explanation for the main results of \cite{hess2010loop} and \cite{franz2021szczarba}.  When $X$ is a $0$-reduced simplicial set, the category  $\widehat{\mathbf{\Omega}}(\widetilde{\mathbf{C}}_*(X))$ has a single object and consequently can be interpreted as a dg bialgebra. Its underlying dg algebra is isomorphic to the extended cobar construction $\widehat{\Omega}C^{\Delta}_*(X)$ introduced in \cite{hess2010loop}. The coalgebra structure on $\widehat{\Omega}C^{\Delta}_*(X)$ is determined by the homotopy Gerstenhaber coalgebra structure of $C^{\Delta}_*(X)$, as discussed in \cite{franz2021szczarba}. The following corollary is then a direct consequence of Theorem 1.2.
\begin{Cor} For any $0$-reduced simplicial set $X$, there are natural quasi-isomorphisms of dg bialgebras  $$\widehat{\mathbf{\Omega}}(\widetilde{\mathbf{C}}_*(X)) \cong \widehat{\Omega}(C^{\Delta}_*(X)) \xrightarrow{\mathbb{T}_X} C_*(\widehat{\mathfrak{C}}(X)) \xrightarrow{C_*(\iota (\mathcal{L} (Sz_X)) \circ \mu_X)}C_*(G^{Kan}(X)),$$ whose composition coincides with the chain level "Szczarba map"  described in section 1.4 of  \cite{hess2010loop}.
\end{Cor} 

Note, however, that $\widehat{\mathbf{\Omega}}$ is now functorial with respect to the initial algebraic chain-level data, in contrast to the extended cobar construction of \cite{hess2010loop} which is functorial with respect to space-level data (or functorial after equipping the initial algebraic data with a choice of basis for the degree $1$ part of the underlying coalgebra). 
\subsection{Organization of the paper}
This paper has five sections and an appendix. In section 2, we recall preliminaries regarding quivers and simplicial sets as well as different adjunctions involving forgetful, free, inclusion and localization functors. In section 3, we construct a natural transformation between the functors $\mathfrak{C}$ and $G$ in terms of a sequence of operators introduced by Szczarba, in a different context, in \cite{szczarba1961homology}. This involves defining these functors as well as giving an interpretation in terms of the framework of necklaces. In section 4, we prove that different localizations of $\mathfrak{C}$ are weakly equivalent to the Kan loop groupoid functor. The end result is that we obtain different explicit combinatorial models (described in terms of simplicial categories) for the path category of a space, as summarized in Theorem \ref{equivalencesofsimplicialcategories}. 
In section 5, we turn to the algebraic problem of modelling the path category, now considered as a category enriched over dg coalgebras, directly from chain-level algebraic structure naturally associated to a simplicial set. This involves introducing the notion of a $B_{\infty}$-coalgebra and a localized version of the many-object cobar construction, building upon \cite{holstein-lazarev}. Finally, in the appendix we discuss different Quillen equivalences of model categories that are used in section 4.

\subsection{Acknowledgments}
MR was supported by NSF Grant 210554 and the Karen EDGE Fellowship. The authors would like to thank Clemens Berger, Kathryn Hess, Julian Holstein, Ralph Kaufmann,  Andrey Lazarev, Yang Mo, and Jim Stasheff for fruitful conversations. We'd also like to thank Kensuke Arawaka for finding typos that lead to correcting the formula for the natural transformation $Sz$ constructed in section 3. 

\section{Quivers, simplicial sets, and localization}
\subsection{Quivers and simplicial sets} 
A \textbf{quiver} (also known as a reflexive directed graph) consists of a pair of sets $Q_1, Q_0$ equipped with maps $t,s: Q_1 \to Q_0$ called the source and target maps and $1: Q_0 \to Q_1$, called the unit map such that for any $q \in Q_0$, $(t \circ 1)(q) = (s \circ 1)(q) = q$. A morphism  of quivers $f:Q \to Q'$ consists on two maps of sets $f_0: Q_0 \to Q'_0$ and $f_1: Q_1 \to Q'_1$ that preserve the unit, source and target maps. Quivers form a category which we denote by $\cat{Quiv}$.

A \textbf{simplicial set} is a functor $X: \mathbf{\Delta}^{op} \to \cat{Set}$, where $\mathbf{\Delta}$ is the category of finite ordinals with order-preserving maps and $\cat{Set}$ is the category of sets. Simplicial sets form a category $\cat{sSet}$ with natural transformations as morphisms. For any simplicial set $X$, we write $X_n=X([n]) \in \cat{Set}.$ The data of a simplicial set $X$ is equivalent to a sequence of sets $\{X_0, X_1, X_2,...\}$ with face maps $d_i=d_{i,n}: X_n \to X_{n-1}$ for $i=0,...,n$ and degeneracy maps $s_j=s_{j,n}: X_n \to X_{n+1}$ for $j=0,...,n$ satisfying the simplicial identities. The category of simplicial sets becomes a monoidal category $(\cat{sSet}, \times)$ when equipped with the level-wise Cartesian product of sets.

Quivers may be thought of as simplicial sets with no non-degenerate simplices of dimension $2$ or above. More precisely, there is a fully faithful functor $$j: \cat{Quiv} \to \cat{sSet}$$ sending a quiver $Q$ to the simplicial set $j(Q)$ defined by $j(Q)_0=Q_0$, $j(Q)_1=Q_1$, with the two face maps $d_0, d_1: j(Q)_1 \to j(Q)_0$ given by the source and target maps, and adjoining degeneracies formally.

We describe several free-forgetful adjunctions that will be used throughout the article. 

\begin{Def} Consider the functor $U: \cat{Cat} \to \cat{Quiv}$, where $\cat{Cat}$ is the category of small categories, given by forgetting the composition rule. The functor $U$ has a left adjoint $F: \cat{Quiv} \to \cat{Cat}$. For any quiver $Q$,  $F(Q)$ is a category, where $\text{Obj}(F(Q)) = Q_0$ and where $\text{Mor}(F(Q))$ is the set of finite formal compositions of elements of $Q_1$. We say that $F(Q)$ is the \textbf{free category} generated by the quiver $Q$ and that $U(\mathcal{C})$ is the \textbf{underlying quiver} of the category $\mathcal{C}$. We obtain an adjunction
$$F: \cat{Quiv} \rightleftarrows \cat{Cat}: U.$$
\end{Def}

\begin{Rem}
We follow the convention of writing the left adjoints on the left hand side with corresponding arrow at the top and going from left to right, just as above. 
\end{Rem}

\begin{Def} Let $\cat{sSet}^*= \left( \Delta^0 \downarrow \cat{sSet} \right)$ be the category of \textbf{pointed simplicial sets}. Denote by 
$\cat{sMon}$ the category of \textbf{simplicial monoids}, i.e. simplicial sets $M$ equipped with an associative and unital map of simplicial sets $M \times M \to M$. 

We now define another free-forgetful adjunction
$$F^*: \cat{sSet}^* \rightleftarrows \cat{sMon}: U^*.$$

Given $(X,b) \in \cat{sSet}^*$, let $F^*(X)_n$ denote the free monoid generated by the set $X_n$ with unit $s_0^{n}(b)$, where $s_0$ denotes the $0$-th degeneracy map. 
Given a simplicial monoid $M$, $U^*(M)$ is the simplicial set obtained by forgetting the product structure, and whose basepoint is given by the unit $e \in M_0$.
\end{Def}

\begin{Def}
We call a set $X$ equipped with a function $X \xrightarrow{\text{deg}} \Z_{\geq 0}$ a \textbf{graded set}. Graded sets form a category $\cat{Set}^{\geq 0}$ with set maps preserving the grading. There is an adjunction $$F^\Delta: \cat{Set}^{\geq 0} \rightleftarrows \cat{sSet}: U^\Delta,$$ where $U^\Delta$ forgets the simplicial face and degeneracy maps, and if $X$ is a graded set, then $F^{\Delta}(X)$ is the simplicial set with $(F^{\Delta}(X))_k = \text{deg}^{-1}(k)$ with faces and degeneracies freely adjoined. In other words,
\begin{equation}
F^{\Delta}(X) \coloneqq \coprod_{k \geq 0} \,  \coprod_{x \in \text{deg}^{-1}(k)} \Delta^k 
\end{equation}
We call $F^{\Delta}(X)$ the \textbf{free simplicial set} generated by the graded set X.
\end{Def}


Let $X$ be a graded singleton set $X = \{ x \}$ with $\text{deg} \, x = k$. Then $x$ corresponds to the unique nondegenerate $k$-simplex $\Delta^k \xrightarrow{x} F^{\Delta}(X) \cong \Delta^k$. We define the poset
\begin{equation}\label{poset bracket}
   \langle x \rangle^k \coloneqq \{ d_1^k x \leq d_1^{k-1} d_0 x \leq \dots \leq d^1_1 d_0^{k-1} x \leq d_0^k x \},
\end{equation}
so that we have natural isomorphisms $\langle x \rangle^k \cong [k]$ and $N \langle x \rangle^k \cong F^\Delta \{ x \} \cong \Delta^k$, where $N$ is the nerve functor.

\begin{Rem}
The elements of $\langle x \rangle^k$ are the vertices of the simplex $x$ ordered from $0$ to $k$.
\end{Rem}

\begin{Def}
A \textbf{graded quiver} is a quiver $Q$ equipped with a function $Q_1 \xrightarrow{\text{deg}} \Z_{\geq 0}$ called the grading. Graded quivers form a category $\cat{Quiv}^{\geq 0}$ with morphisms of quivers that preserve the gradings. 

Let $\cat{Cat}_{\cat{sSet}}$ denote the category of small categories enriched over $(\cat{sSet}, \times)$. This is equivalently the category of simplicial objects in $\cat{Cat}$ whose simplicial operators are the identity on objects. We refer to these objects as \textbf{simplicial categories}. If $\cat{C}$ is a simplicial category, then $\cat{C}(x,y)$ is a simplicial set for all objects $x,y \in \text{Obj}(\cat{C})$, and thus we refer to the set $\cat{C}(x,y)_k$ as the set of \textbf{$k$-morphisms} with source $x$ and target $y$.

There is a forgetful functor
$$\mathcal{U}: \cat{Cat}_{\cat{sSet}} \to \cat{Quiv}^{\geq 0}$$ 
that forgets about composition and the simplicial structure of the morphisms. The forgetful functor has a left adjoint $$\mathcal{F}: \cat{Quiv}^{\geq 0} \to \cat{Cat}_{\cat{sSet}}.$$ 
For any graded quiver $Q$, $\mathcal{F}(Q)$ is the simplicial category having $Q_0$ as objects and for each $x,y \in Q_0$, the simplicial set of morphisms $\mathcal{F}(Q)(x,y) \in \cat{sSet}$ is defined as follows. First consider $F^\Delta (Q_1(x,y))$, the free simplicial set generated by the graded set $Q_1(x,y)$, for each $x,y \in Q_0$, where $Q_1(x,y)=\{e \in Q_1: s(e)=x, t(e)=y\}$. Then define $\mathcal{F}Q(x,y)_n$ to consist of all finite formal compositions $e_k \cdot \dotsc \cdot e_1$ where $e_i \in F^\Delta(Q_1(s(e_i),t(e_i))_n$ and $t(e_i)=s(e_{i+1})$ for all $i=1,\dots,n$. The face maps $d_j: \mathcal{F}(Q)(x,y)_n \to \mathcal{F}(Q)(x,y)_{n-1}$ are defined by $d_j(e_k \cdot \dotsc \cdot e_1) = d_j(e_k) \cdot \dotsc \cdot d_j(e_1)$ and degeneracies are defined similarly. The simplicial category $\mathcal{F}(Q)$ is called the \textbf{free simplicial category} generated by the graded quiver $Q$.

\end{Def}

\subsection{Localization} 
\begin{Def} \label{localizationofcategories} Let $L: \cat{Cat} \to \cat{Gpd}$ be the classical (Gabriel-Zisman) \textbf{localization} functor from categories to groupoids. Given any category $C$, $L(C)$ is the groupoid whose morphisms are obtained by adding formal inverses to all morphisms in $C$. The localization functor is the left adjoint of an adjunction
$$L: \cat{Cat} \rightleftarrows \cat{Gpd}: i$$
where $i$ is the fully faithful inclusion from groupoids into categories.

The above adjunction induces a new adjunction
$$\mathcal{L}: \cat{Cat}_{\cat{sSet}} \rightleftarrows \cat{Gpd}_{\cat{sSet}}: \iota$$
between simplicial categories and simplicial groupoids. For any simplicial category $\mathcal{C}$, we define $\mathcal{L}(\mathcal{C})_n = L( \mathcal{C}_n )$ together with the obvious face and degeneracy maps on morphisms. In other words, we apply the localization functor $L$ degree-wise. We also call $\mathcal{L}$ the localization functor. For any simplicial category $\mathcal{C}$ the unit of the adjunction gives us a natural map of simplicial categories $\mathcal{C} \to \iota \mathcal{L}(\mathcal{C})$, which we often write as $\mathcal{C} \to \mathcal{L}(\mathcal{C})$

For any map $i: \mathcal{W} \xrightarrow{} \mathcal{C}$ of simplicial categories, we denote by $\mathcal{C}[\mathcal{W}^{-1}]$ the (ordinary) pushout of the maps $i: \mathcal{W} \to \mathcal{C}$ and $\mathcal{W} \to \mathcal{L}(\mathcal{W})$ in $\cat{Cat}_{\cat{sSet}}$ as given by the following diagram:
\begin{equation}
    \begin{tikzcd}
\mathcal{W} \arrow[d, "i"'] \arrow[r]  & {\mathcal{L}(\mathcal{W})} \arrow[d] \\
\mathcal{C} \arrow[r]                      & {\mathcal{C}[\mathcal{W}^{-1}]} \arrow[ul, phantom, "\ulcorner", very near start].    
\end{tikzcd}
\end{equation}
\end{Def}
The following is now straightforward to verify.
\begin{Prop}
The adjunction $\mathcal{L}: \cat{Cat}_{\cat{sSet}, B} \rightleftarrows  \cat{Gpd}_{\cat{sSet},DK}: \iota$ is a Quillen adjunction when these categories are equipped with the Bergner model structure (Theorem \ref{scatB}) and Dwyer-Kan model structure (Theorem \ref{sgpdDK}), respectively. 
\end{Prop}

By restricting the functors $\mathcal{L}$ and $\iota$ to categories with one object we obtain an induced adjunction $\mathcal{L}: \cat{sMon} \rightleftarrows \cat{sGrp}: \iota$. Denote by $\cat{sSet}^*_{KQ}$ the model category given by the slice model structure on  $(\Delta^0 \downarrow \cat{sSet}) =\cat{sSet}^*$ induced by the Kan-Quillen model structure on $\cat{sSet}.$

\begin{Prop} \label{locfree}
The adjunction $\mathcal{L}F^*: \cat{sSet}^*_{KQ} \rightleftarrows \cat{sGrp}_K: U^*\iota$ is a Quillen adjunction. 
\end{Prop}
\begin{proof} This is immediate since the right adjoint $U^*\iota$ preserves fibrations and acyclic fibrations (in fact, it preserves all weak equivalences). 
\end{proof}

\section{The Szczarba map}

In this section we describe a natural transformation $Sz: \mathfrak{C} \xRightarrow{ } G$, where $$\mathfrak{C}: \cat{sSet} \to \cat{Cat}_{\cat{sSet}}$$ is the \textit{rigidification} functor, the left adjoint of Cordier's homotopy coherent nerve functor $\mathfrak{N}: \cat{Cat}_{\cat{sSet}} \to \cat{sSet}$ (\cite{Cor82}, \cite{lurie2006higher}, \cite{DuSp11}) and $$G: \cat{sSet} \to \cat{Cat}_{\cat{sSet}}$$ is the left adjoint of the \textit{classifying space} functor $\overline{W}: \cat{Cat}_{\cat{sSet}} \to \cat{sSet}$ \cite{dwyer_kan_1984}. The localization $\mathcal{L} \circ G$ is known as the \textit{Kan loop groupoid} functor.  The natural transformation $Sz: \mathfrak{C} \xRightarrow{ } G$ is described in terms of a sequence of simplicial operators that are reminiscent of those introduced by Szczarba \cite{szczarba1961homology}.

\subsection{The Kan loop groupoid functor}

\begin{Rem}
The following definition comes from \cite[Section 2.6]{hinich2007homotopy}. 
\end{Rem}

\begin{Def} \label{GHin def}
Let $G(\Delta^n) = \mathcal{F} (Q^n)$ denote the free simplicial category generated by the graded quiver $Q^n$ with $Q^n_0=\{0,...,n\}$ and where the graded set $Q^n_1$ has exactly one element $g_i$ in degree $n-i$ with $s(g_i) = i-1, t(g_i) = i$ for $i = 1, \dots, n$. 

It follows that every simplex in $G(\Delta^n)(i,j)$ is a formal composition of the elements $g_k$, their faces, and their degeneracies. In other words, using the notation introduced in \ref{poset bracket}, $$G(\Delta^n)(i,j) \cong \Delta^{n-j} \times \dots \times \Delta^{n-i-1} \cong N(\langle g_j \rangle^{n-j} \times \dots \times \langle g_{i+1} \rangle^{n-i-1}),$$ where $N$ denotes the nerve functor. The assignment $[n] \mapsto G (\Delta^n)$ defines a cosimplicial simplicial category as follows.

Let $d^i: G(\Delta^n) \to G(\Delta^{n+1})$ with $1 \leq i \leq n+1$ act on objects in the obvious way and define it on the generators by:
\begin{equation}
    d^i(g_j) = \begin{cases}
    g_{j+1} & \text{if } i < j, \\
    g_{i+1} d_0 g_i & \text{if } i = j,\\
    d_{i-j} g_j & \text{if } i > j.\\ 
    \end{cases} \\
\end{equation}
Similarly let $s^i: G(\Delta^n) \to G(\Delta^{n-1})$ act on objects in the obvious way and define it on the generators by:
\begin{equation}
s^i(g_j) = \begin{cases}
    s_{i-j} g_j & \text{if } j \leq i \\
    \text{id}_j & \text{if } j = i + 1 \\
    g_{j-1} & \text{if } j > i + 1.
    \end{cases}
\end{equation}
where $\text{id}_j$ denotes the unique $(n-j)$-simplex in $G(\Delta^{n-1})(j,j) \cong \Delta^0$.

The category $\cat{Cat}_{\cat{sSet}}$ is cocomplete, thus we can define the functor $$G: \cat{sSet} \to \cat{Cat}_{\cat{sSet}}$$ on any simplicial set $X$ by the formula  $$G(X):= \underset{\Delta^n \to X}{\text{colim }} G(\Delta^n) \in \cat{Cat}_{\cat{sSet}}.$$ 
\end{Def}

\begin{Def}
Define the functor $$\conj{W}: \cat{Cat}_{\cat{sSet}} \to \cat{sSet}$$ on any simplicial category $C$ by letting $\conj{W}(C) \in \cat{sSet}$ be the simplicial set having as $n$-simplices the set $\conj{W}(C)_n= \cat{Cat}_{\cat{sSet}}(G (\Delta^n), C)$ and face and degeneracy maps induced by the cosimplicial structure of $G (\Delta^{\bullet})$.

Hence we have an adjunction 
\begin{equation}
    G : \cat{sSet} \rightleftarrows \cat{Cat}_{\cat{sSet}}: \conj{W}.
\end{equation}
\end{Def}

\begin{Def}
We now recall the definition of the \textbf{Kan loop groupoid} functor $G^{Kan}: \cat{sSet} \to \cat{Gpd}_{\cat{sSet}}$. If $X$ is a simplicial set, then for each integer $n \geq 0$, define $G^{Kan}(X)_n$ to be the free groupoid with object set $\text{Obj}( G^{Kan}(X)_n) = \lbrace \conj{x} \, | \, x \in X_0 \rbrace$ and morphism set $\text{Mor}(G^{Kan}(X)_n)$ generated by elements $\conj{y}: \conj{s(y)} \to \conj{t(y)}$, where $y \in X_{n+1}$ and $s = (d_1)^{n+1}, t = d_0(d_2)^n$ with relation $\conj{s_0 z} = \text{id}_{\conj{s(z)}}$ for $z \in X_n$. Define face maps $\delta_i: G^{Kan}(X)_n \to G^{Kan}(X)_{n-1} $ and degeneracy maps $\sigma_i: G^{Kan}(X)_n \to  G^{Kan}(X)_{n+1}$ to be the identity on objects and on morphisms given by:
\begin{equation}
\begin{aligned}
& \sigma_i \conj{x} = \conj{s_{i+1} x} \\
& \delta_i\conj{x} = \conj{d_{i+1} x} \text{ for } 1 \leq i \leq n \\
& \delta_0 \conj{x} = (\conj{d_0 x})^{-1} \conj{d_1 x }.
\end{aligned}
\end{equation}
\end{Def}

\begin{Def}
Let $\overline{W}^{Kan}: \cat{Gpd}_{\cat{sSet}} \to \cat{sSet}$ denote the functor such that:
\begin{equation*}
\begin{aligned}
    & \overline{W}^{Kan}(\mathcal{G})_0 = \text{Obj } \mathcal{G}_0 \\
    & \overline{W}^{Kan}(\mathcal{G})_n  = \lbrace (h_{n-1}, \dots, h_0) \, | \, h_k \in \text{Mor }\mathcal{G}_k, \, t(h_k) = s(h_{k-1}) \rbrace \\
    \end{aligned}
    \end{equation*}
    with face and degeneracy maps given by:
    \begin{equation*}
    d_i(h_{n-1}, \dots, h_0)  = \begin{cases} (h_{n-2}, \dots, h_0), & i = 0 \\
    (d_{i-1} h_{n-1}, \dots,h_{n-i-1} d_0 h_{n-i}, \dots, h_0), & 0 < i < n\\
    (d_{n-1} h_{n-1}, \dots, d_1 h_1), & i = n \\
    \end{cases}
    \end{equation*}
    \begin{equation*}
    s_i(h_{n-1}, \dots, h_0)= \begin{cases} (\text{id}_{s(h_{n-1})}, h_{n-1}, \dots, h_0), & i = 0 \\
    (s_{i-1} h_{n-1}, \dots, s_0 h_{n-i},\text{id}_{t(h_{n-i})},h_{n-i-1}, \dots, h_0), & 0 < i. \\
    \end{cases}
\end{equation*}
\end{Def}

\begin{Prop}\label{prop G Kan iso to L G}
There is a natural isomorphism of functors 
\begin{equation}
G^{Kan} \cong \mathcal{L} \circ G,
\end{equation}
where $\mathcal{L}$ denotes the localization functor, i.e. the left adjoint  to the inclusion of simplicial groupoids into simplicial categories.
\end{Prop}

\begin{proof} 
Since  $\conj{W}^{Kan}$ and $G^{Kan}$ are adjoint functors, as proven in \cite{dwyer_kan_1984}, this proposition follows by comparing the definitions of $\conj{W} \circ \iota$ and $\conj{W}^{Kan}$ and taking adjoints.

For completeness, we describe the isomorphism directly. Define $$\eta_{\Delta^n}: G(\Delta^n) \xrightarrow{} G^{Kan}(\Delta^n)$$ to be the identity on objects, and for any $i \in [n]$, $$\eta_{\Delta^n}: G(\Delta^n)(i-1,i) \xrightarrow{} G^{Kan}(\Delta^n)(i-1,i)$$ is determined by 
\begin{equation}
    g_i \mapsto \conj{[i-1 \dots n]},
\end{equation}
where $[i \dots j]$ denotes the $(j-i)$-simplex in $\Delta^n$ determined by the vertices $\{ i, i+1, \dots, j \} \subseteq [n]$.  By the universal property of the localization functor, $\eta_{\Delta^n}$ gives rise to an isomorphism $$\widetilde{\eta_{\Delta^n}}: \mathcal{L}G(\Delta^n) \to G^{Kan}(\Delta^n)$$ of simplicial groupoids for all $n \geq0$. These maps give rise to an isomorphism of cosimplicial simplicial groupoids $$\widetilde{\eta_{\Delta^{\bullet}}}:\mathcal{L}G(\Delta^{\bullet}) \to G^{Kan}(\Delta^{\bullet})$$ which induces a natural isomorphism $\mathcal{L} \circ G \xRightarrow{} G^{Kan}$.

\end{proof}
\subsection{The rigidification functor}
\begin{Def}\label{C def}
Given a standard $n$-simplex $\Delta^n$, define a simplicial category $\mathfrak{C}(\Delta^n)$ as follows:
\begin{enumerate}
    \item $\text{Obj} \, \mathfrak{C}(\Delta^n) = [n] = \lbrace 0, 1, \dots, n \rbrace$.
    \item If $i,j \in [n]$, then $\mathfrak{C}(\Delta^n)(i,j) = \begin{cases}
    \emptyset & \text{if }i > j \\
    N(P_{i,j}) \cong (\Delta^1)^{ \times (j - i -1)} & \text{if } i < j\\
    \Delta^0 & \text{if } i=j 
    \end{cases}$\\
    where $N$ is the nerve functor and, for $i<j$,  $P_{i,j}$ is the poset whose elements are subsets \[U = \{i, i_0, \dots, i_m, j \} \subseteq \{i,i+1, \dots, j-1,j\}\]
    and $U \leq V$ if $V \subseteq U$.
    \item If $i_0 \leq \dots \leq i_{m}$, then the composition
    \begin{equation*} 
        \mathfrak{C}(\Delta^n)(i_{m-1},i_{m}) \times \dots \times \mathfrak{C}(\Delta^n)(i_0,i_1) \to \mathfrak{C}(\Delta^n)(i_0,i_m)
    \end{equation*}
    is induced by the map of the posets
    \begin{equation*}
    \begin{aligned}
        & P_{i_{m-1},i_{m}} \times \dots \times P_{i_0,i_1} \to P_{i_0, i_{m}} \\
        &(U_{m-1}, \dots, U_0) \mapsto U_{m-1} \cup \dots \cup U_0.
        \end{aligned}
    \end{equation*}
\end{enumerate}

Each $0$-morphism in $\mathfrak{C}(\Delta^n)(i,j)$ corresponds to a subset $\lbrace i, i_0, \dots, i_m, j \rbrace \subset [n]$ with $i<i_0< \dots <i_m<j$. Hence, each $0$-morphism can be written as a composition $\lbrace i_m, j \rbrace  \times \dots \times \lbrace i, i_0 \rbrace \mapsto \lbrace i, i_0, \dots, i_m, j \rbrace$. Each set $(\mathfrak{C}(\Delta^n)(i,j))_0 =N(P_{i,j})_0$ has a special element $\{i,j\}$, we call $0$-morphisms of this form \textbf{indecomposable}. It follows that any $0$-morphism is a unique composition of indecomposable $0$-morphisms.

Each $k$-simplex in $\mathfrak{C}(\Delta^n)(p,q)$, where $0 \leq k \leq q - p - 1$ is given by a sequence $$\{ p, q \} \geq \{ p, i_1, q \} \geq \{ p, i_1, i_2, q \} \geq \dots \geq \{p, i_1, \dots, i_k, q \}$$
which we can write as the sequence $(i_1, \dots, i_k)$. The empty sequence, denoted $\emptyset$ corresponds to the indecomposable $\{p , q \}$.
Furthermore, the assignment $[n] \mapsto \mathfrak{C} (\Delta^n)$ defines a cosimplicial object in simplicial categories with coface and codegeneracy maps which act in the obvious way on objects and on simplices by
\begin{equation}
    \begin{aligned}
        & d^j: \mathfrak{C}(\Delta^n) \to \mathfrak{C}(\Delta^{n+1}) \\
        & d^j(i_1, \dots, i_k) = (d^j(i_1), \dots, d^j(i_k)) \\
        & s^j: \mathfrak{C}(\Delta^n) \to \mathfrak{C}(\Delta^{n-1}) \\
        & s^j(i_1, \dots, i_k) = (s^j(i_1), \dots, s^j(i_k))).
    \end{aligned}
\end{equation}
For any simplicial set $X \in \cat{sSet}$, define 
$$\mathfrak{C}(X)= \underset{\sigma: \Delta^n \to X}{\text{colim}} \mathfrak{C}(\Delta^n).$$
This construction gives rise to a functor $\mathfrak{C}: \cat{sSet} \to \cat{Cat}_{\cat{sSet}}$ called the \textbf{rigidification functor}. 
\end{Def}

\begin{Rem}
The functor $\mathfrak{C}$ given above is precisely \cite[Definition 1.1.5.1]{lurie2006higher} but with opposite underlying posets $P_{i,j}$, and agrees with the definition given in \cite[Section 2.2]{hinich2007homotopy}.
\end{Rem}

\begin{Def} \label{homotopy coherent nerve}
The \textbf{homotopy coherent nerve} $\mathfrak{N}: \cat{Cat}_{\cat{sSet}} \to \cat{sSet}$ is given by \begin{equation}
    (\mathfrak{N} C)_n \coloneqq \cat{Cat}_{\cat{sSet}}(\mathfrak{C}(\Delta^n), C),
\end{equation}
on any $C \in \cat{Cat}_{\cat{sSet}}$. This provides us with an adjunction
\begin{equation} \label{homotopy coherent adjunction}
    \mathfrak{C}: \cat{sSet} \rightleftarrows \cat{Cat}_{\cat{sSet}} : \mathfrak{N},
\end{equation}
which is actually a Quillen equivalence when $\cat{sSet}$ is equipped with Joyal's model structure and $\cat{Cat}_{\cat{sSet}}$ with Bergner's model structure, see Theorem \ref{rigidification} in the Appendix. 
\end{Def}

\subsection{Rigidification in terms of necklaces}\label{necklaces}

We recall a description of the mapping spaces $\mathfrak{C}(X)(x,y)$ given in \cite{DuSp11} in terms of the framework of necklaces. 

\begin{Def}

A \textbf{necklace} is simplicial set of the form 
$$N=\Delta^{n_1} \vee ... \vee \Delta^{n_k}$$ 
obtained from an ordered sequence of standard simplices with $n_i>0$ for $i=1,...,k$, by identifying the final vertex of one to the first vertex of its successor. The standard simplex $\Delta^{n_i}$ in the sequence is called the $i$-th \textbf{bead} of the necklace $N$. Define $\text{dim}(N)\coloneqq n_1+...+n_k-k$. We consider $\Delta^0$ as a necklace of dimension $0$. 

Any necklace $N$ has a natural ordering on its vertices given by the ordering of the beads of $N$ and the ordering of the vertices on each bead. Denote by $\alpha_N$ and $\omega_N$ the first and last vertices of $N$. Necklaces are the objects of a category $\cat{Nec}$ with morphisms being maps of simplicial sets $f: N \to N'$ such that $f(\alpha_N)=\alpha_{N'}$ and $f(\omega_N)=\omega_{N'}$. For $X \in \cat{sSet}$ and $x,y\in X_0$ denote by $(\cat{Nec} \downarrow X)_{x,y}$ the full subcategory of the over category $(\cat{Nec} \downarrow X)$ consisting of those maps $f: N\to X$ such that $f(\alpha_N)=x$ and $f(\omega_N)=y$. 

The category $\cat{Nec}$ of necklaces has a non-symmetric monoidal structure 
$$\vee: \cat{Nec} \times \cat{Nec} \to \cat{Nec}$$
given by concatenating necklaces. The unit object in the monoidal structure is $\Delta^0.$

The morphisms of $\cat{Nec}$ are generated through the monoidal structure by the following four types of morphisms:

\begin{enumerate}
	\item $\partial_j \colon \Delta^{n-1} \hookrightarrow \Delta^n$ for $j = 1, \dots, n-1$,
	\item $\Delta_{[j], [n-j]} \colon \Delta^{j} \vee \Delta^{n-j} \hookrightarrow \Delta^n$ for $j = 1, \dots, n-1$,
	\item $s_j \colon \Delta^{n+1} \twoheadrightarrow \Delta^{n}$ for $j = 0, \dots, n$ and $n>0$, and
	\item $s_0 \colon \Delta^1 \twoheadrightarrow \Delta^0$.
	\end{enumerate}
\end{Def}

The mapping spaces $\mathfrak{C}(X)(x,y)$ are obtained by gluing simplicial cubes labeled by necklaces in $X$ from $x$ to $y$. 

\begin{Prop}[{\cite{DuSp11}}]\label{dsnecklaces}
For any $X \in \cat{sSet}$ and $x,y \in X$, there are natural isomorphisms of simplicial sets
$$\mathfrak{C}(X)(x,y) \cong \underset{(f: N \to X) \in (\cat{Nec} \downarrow X)_{x,y}}{\emph{colim}} \mathfrak{C}(N)(\alpha_N, \omega_N) \cong \underset{(f: N \to X) \in (\cat{Nec} \downarrow X)_{x,y}}{\emph{colim}}  (\Delta^{1})^{\times\emph{dim}(N)}.$$ Furthermore, the composition $\mathfrak{C}(X)(y,z) \times \mathfrak{C}(x,y) \to \mathfrak{C}(X)(x,z)$ is induced by the concatenation of necklaces $\vee: \cat{Nec} \times \cat{Nec} \to \cat{Nec}$.
\end{Prop}


\subsection{The natural transformation $Sz$} 

In this section we will define a map $Sz: \mathfrak{C}(\Delta^\bullet) \to G(\Delta^\bullet)$ of cosimplicial simplicial categories.

Given non-negative integers $p$, $q$ and $\ell$, let
\begin{equation*}
    S^\ell_{p,q} = \{ i = (i_1, \dots, i_\ell) \in \{ p + 1, \dots, q -1 \}^{\times \ell} \; : \: i_r \neq i_s, \text{ for } r \neq s \}.
\end{equation*}
For $\ell = 0$, we set $S^0_{p,q} = \{ \emptyset \}$, and we call $\emptyset$ the empty sequence.

For a fixed $n \geq 1$, let $0 \leq p < q \leq n$ and $0 \leq \ell \leq q-p-1$. Consider the set $\text{nd}(\mathfrak{C}(\Delta^n)(p,q)_\ell)$ of non-degenerate $\ell$-simplices of $\mathfrak{C}(\Delta^n)(p,q)$. There is an obvious bijection
$\text{nd}(\mathfrak{C}(\Delta^n)(p,q)_\ell) \cong S^\ell_{p,q}$. Thus we implicitly identify sequences $(i_1, \dots, i_\ell)$ as above with non-degenerate $\ell$-simplices.

Now if $i = (i_1, \dots, i_{\ell -1})$ is a non-degenerate $(\ell-1)$-simplex, and $i_\ell \in \{p+1, \dots, q  - 1\} \setminus \{i_1, \dots, i_{\ell -1}\}$, then let $\omega_{(i_1, \dots, i_{\ell -1})}(i_{\ell})$ denote the largest integer in $\{p, i_1, \dots, i_{\ell - 1}, q \}$ such that $\omega_i(i_\ell) < i_\ell$.

Given a sequence $i = (i_1, \dots, i_{\ell - 1}, i_\ell)$, let $i' = (i_1, \dots, i_{\ell - 1})$ and $i^{(k)} = (i_1, \dots, i_{\ell - k})$ for $1 \leq k \leq \ell$ and $1 \leq \ell \leq q-p-1$, where $i^{(\ell)} = \emptyset$. Now for a fixed $n \geq 1$, $0 \leq p < q \leq n$, $p + 1 \leq k \leq q$ and $0 \leq \ell \leq q - p - 1$, define a function
\begin{equation*}
    \alpha_k : S^\ell_{p,q} \to \{0, \dots, n - k \}
\end{equation*}
as follows. If $\ell = 0$, so that $i = \emptyset$, then set $\alpha_k(\emptyset) = q - k$. For nonempty $i = (i_1, \dots, i_\ell)$, define $\alpha_k(i)$ inductively by
\begin{equation}
   \alpha_k(i) = 
   \begin{cases}
  i_\ell - k, & \omega_{i'}(i_\ell) < k \leq i_\ell \\
  \alpha_k(i'), & k \leq \omega_{i'}(i_\ell) \text{ or } i_\ell < k.
    \end{cases}
\end{equation}

\begin{Cons} \label{Sz construction}
We construct a map of simplicial categories $$Sz_{\Delta^n}: \mathfrak{C}(\Delta^n) \to G(\Delta^n)$$ as follows. On objects, $Sz$ is the identity map. For any nondegenerate $\ell$-simplex in $\mathfrak{C}(\Delta^n)(p,q)_\ell$, determined by an ordered sequence $i = (i_1, \dots, i_\ell)$, let $Sz_{\Delta^n}(i)$ be the $\ell$-simplex in 
$$G(\Delta^n)(p,q) \cong N(\langle g_q \rangle^{n-q}) \times \dots \times N(\langle g_{p+1}^{n-(p+1)} \rangle)$$
given by
\begin{equation} \label{Sz formula}
Sz_{\Delta^n}(i) =  \left( \mathcal{E}_{i,q} \, g_q, \, \mathcal{E}_{i,q-1} \, g_{q-1}, \, \dots \, , \, \mathcal{E}_{i, p+2} \, g_{p + 2}, \, \mathcal{E}_{i, p+1} \, g_{p+1} \right),
\end{equation}where the $\mathcal{E}_{i, k}$ for $p+1 \leq k \leq q$ are simplicial operators which we now define. First, given any simplicial operator $\tau = s_{i_0} \dots s_{i_a} d_{j_0} \dots d_{j_b}$, we write $\tau' = s_{i_0 + 1} \dots s_{i_a + 1} d_{j_0 + 1} \dots d_{j_b + 1}$ and denote the iteration of this operation by $\tau'' = (\tau')' = s_{i_0 + 2} \dots d_{j_b + 2}$, and more generally $\tau^{(m)} = s_{i_0 + m} \dots d_{j_b + m}$. Define the operators $\mathcal{E}_{i,k}$ by induction on $\ell$, the length of $i$. For the empty sequence we define
\begin{equation*}
        \mathcal{E}_{\emptyset, k} = d_1^{n-q} \, d_0^{q-k},
\end{equation*}
and for $i \neq \emptyset$ we define
\begin{equation}
    \mathcal{E}_{i,k} = 
    \begin{cases}
        s_0 \, \mathcal{E}_{i',k} & \text{if } \alpha_k(i') = \alpha_k(i) \\
        \mathcal{E}'_{i', k} \, s_0^{\alpha_{k}(i) + 1} \, d_0^{\alpha_k(i)} & \text{if } \alpha_k(i) < \alpha_k(i').
    \end{cases}
\end{equation}
Note that by definition it is not possible for $\alpha_k(i) > \alpha_k(i')$.
\end{Cons}

\begin{Rem}
These simplicial operators $\mathcal{E}_{i,k}$ 
are reminiscent of the operators appearing in \cite[Theorem 2.1]{szczarba1961homology}.
\end{Rem} 

\begin{Prop} \label{prop Sz map of cosimplicial simplicial categories}
Construction \ref{Sz construction} determines a map of cosimplicial simplicial categories $Sz_{\Delta^{\bullet}}: \mathfrak{C}(\Delta^{\bullet}) \xrightarrow{} G(\Delta^{\bullet})$ and consequently a natural transformation 
$$Sz: \mathfrak{C} \xRightarrow{} G.$$
\end{Prop}

\begin{proof}
Recall the adjunction $\tau_1 \dashv N: \cat{sSet} \rightleftarrows \cat{Cat}$, given by the ordinary nerve and fundamental category, is strong monoidal and induces an adjunction $\tau^{\cat{Cat}}_1 \dashv N^{\cat{Cat}}: \cat{Cat}_{\cat{sSet}} \rightleftarrows \cat{Cat}_{\cat{Cat}}$ \cite[Digression 1.4.2]{riehl2018elements}. Furthermore, $N^{\cat{Cat}}$ is fully faithful.
We show $Sz_{\Delta^n} = N^{\cat{Cat}}(\text{Hin})$ for a map $\text{Hin}: P_{\mathfrak{C}}(\Delta^n) \to P_{G}(\Delta^n)$ between poset enriched categories, from which it will follow that $Sz_{\Delta^n}$ is compatible with compositions and consequently defines a map of simplicial categories. 

Let $P_{\mathfrak{C}}(\Delta^n)$ be the underlying poset enriched category of $\mathfrak{C}(\Delta^n)$. Namely, $P_{\mathfrak{C}}(\Delta^n)(i,j)=P^n_{i,j}$ if $i < j$, as given in Definiton \ref{C def}. Similarly, let $P_G(\Delta^n)$ be the underlying poset enriched category of $G(\Delta^n)$. Namely, $$P_G(\Delta^n)(i,j)=\langle g_j \rangle^{n-j} \times \cdots \times \langle g_{i+1} \rangle^{n-(i+1)}$$
for $i < j$, and $P_G(\Delta^n)(j,j) = \{ \text{id}_j \}$, the trivial poset, as given in Definition \ref{GHin def}. We have $N^{\cat{Cat}}(P_{\mathfrak{C}}(\Delta^n))=\mathfrak{C}(\Delta^n)$ and $N^{\cat{Cat}}(P_G(\Delta^n))= G(\Delta^n)$.

Define $\text{Hin} : P_{\mathfrak{C}}(\Delta^n) \to P_{G}(\Delta^n)$ to be the identity on objects. The objects in each poset $P_{\mathfrak{C}}(\Delta^n)(p,q)$ can be written uniquely as a composition of indecomposable $0$-morphisms. Therefore it is sufficient to define $\text{Hin} : P_{\mathfrak{C}}(\Delta^n)(p,q) \to P_{G}(\Delta^n)(p,q)$ on indecomposable $0$-morphisms. Given the unique indecomposable $\{p, q\} \in P_{\mathfrak{C}}(\Delta^n)(p,q)$, let
\begin{equation} \label{eqn hinich map}
\text{Hin}(\{p,q\}) = (d_1^{n-q} g_q, \, d_1^{n-q} \, d_0 \, g_{q-1}, \, \dots \, , \, d_1^{n-q} \, d_0^{q-p-1} \, g_{p+1}).\footnote{this map appears as $\psi$ in \cite[Section 2.6.1]{hinich2007homotopy}}
\end{equation}

We may extend this to a map on all morphisms by requiring it to be compatible with compositions. The map $\text{Hin}$ preserves the poset structures. To verify this it is sufficient to note
\begin{equation*}
    \text{Hin}(\{ p , q \}) \geq \text{Hin}(\{ p, k , q \}) = \text{Hin}(\{k,q \}) \text{Hin}(\{p, k \})
\end{equation*}
for $p < k < q$. This follows from
\begin{equation*}
    \text{Hin}(\{ p, k , q\}) = (d_1^{n-q} g_q, \, \dots \, , d_1^{n-q} \, d_0^{q - (k + 1)} g_{k+1}, \, d_1^{n-k} g_k, \,  \dots \, , \, d_1^{n-k} d_0^{k-(p + 1)} g_{p+1}),
\end{equation*}
together with the fact that the poset structure $P_G(\Delta^n)(p,q)$ is given by the product of posets and every component of $\text{Hin}(\{p, k, q \})$ is less than the corresponding component of $\text{Hin}(\{p, q \})$. This defines a map of cosimplicial poset-enriched categories $\text{Hin} : P_\mathfrak{C}(\Delta^\bullet) \to P_G(\Delta^{\bullet})$, which after applying $N^{\cat{Cat}}$ we obtain a map of cosimplicial simplicial categories. From a straightforward computation, one can check that $Sz \cong N^{\cat{Cat}} (\text{Hin})$.
\end{proof}

\begin{Ex}
Consider the $2$-simplex in $\mathfrak{C}(\Delta^3)(0,3)_2$ given by 
$$
\{0,3\} \geq \{0,2,3\} \geq \{0,1,2,3\}.
$$
This simplex corresponds to the sequence $i = (2,1)$. So with $n = 3, \, p = 0, \, q = 3$, we compute
\begin{equation*}
    \alpha_3(\emptyset) = 0, \; \alpha_2(\emptyset) = 1, \; \alpha_1(\emptyset) = 2
\end{equation*}
\begin{equation*}
   \omega_\emptyset(2) = 0, \; \alpha_3(2) = 0, \; \alpha_2(2) = 0, \; \alpha_1(2) = 1 
\end{equation*}
\begin{equation}
\omega_{(2)}(1) = 0, \; \alpha_3(2,1) = 0, \; \alpha_2(2,1) = 0, \; \alpha_1(2,1) = 0.
\end{equation}
With this we can then compute
\begin{equation}
    \begin{aligned}
    Sz_{\Delta^3}(i) & = (\mathcal{E}_{(2,1), 3} \, g_3, \, \mathcal{E}_{(2,1),2} \, g_2, \, \mathcal{E}_{(2,1),1} \, g_1) \\
    & = (s_0 \mathcal{E}_{(2),3} \, g_3, \, s_0 \mathcal{E}_{(2),2} \, g_2, \, \mathcal{E}'_{(2), 1} s_0 g_1) \\
    & = (s_0^2 \mathcal{E}_{\emptyset, 3} \, g_3, \, s_0 \mathcal{E}'_{\emptyset, 2} s_0 g_2, \, \mathcal{E}''_{\emptyset, 1} s_1^2 d_1 s_0 g_1)  \\
    & = (s_0^2 g_3, \, s_0 d_1 s_0 g_2, \, d_2^2 s_1^2 g_1) \\
    & = (s_0^2 g_3, \, s_0 g_2, \, g_1).
    \end{aligned}
\end{equation} 
A similar computation gives 
\begin{equation*}
    Sz_{\Delta^3}(1,2) = (s_0^2 g_3, \, s_1 g_2, \, s_0 d_1 g_1).
\end{equation*}
\end{Ex}

We include a diagram illustrating the map $Sz_{\Delta^3} : \mathfrak{C}(\Delta^3)(0,3) \to G(\Delta^3)(0,3)$.

\vspace{.5cm}

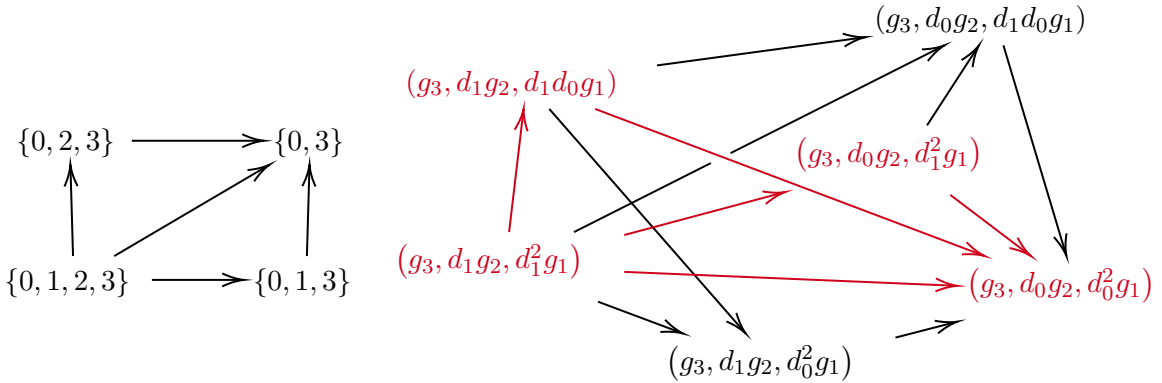
\tikzset{every picture/.style={line width=0.75pt}} 

\begin{tikzpicture}[x=0.75pt,y=0.75pt,yscale=-1,xscale=1]

\draw (155,132.4) node [anchor=north west][inner sep=0.75pt]    {$\{0,3\}$};
\draw (145,202.4) node [anchor=north west][inner sep=0.75pt]    {$\{0,1,3\}$};
\draw (21,202.4) node [anchor=north west][inner sep=0.75pt]    {$\{0,1,2,3\}$};
\draw (27,132.4) node [anchor=north west][inner sep=0.75pt]    {$\{0,2,3\}$};
\draw (215.67,190.35) node [anchor=north west][inner sep=0.75pt]  [color={rgb, 255:red, 208; green, 2; blue, 27 }  ,opacity=1 ]  {$\left( g_{3} ,d_{1} g_{2} ,d_{1}^{2} g_{1}\right)$};
\draw (221,102.4) node [anchor=north west][inner sep=0.75pt]  [color={rgb, 255:red, 208; green, 2; blue, 27 }  ,opacity=1 ]  {$( g_{3} ,d_{1} g_{2} ,d_{1} d_{0} g_{1})$};
\draw (351,241.4) node [anchor=north west][inner sep=0.75pt]    {$\left( g_{3} ,d_{1} g_{2} ,d_{0}^{2} g_{1}\right)$};
\draw (455,70.4) node [anchor=north west][inner sep=0.75pt]    {$( g_{3} ,d_{0} g_{2} ,d_{1} d_{0} g_{1})$};
\draw (415,136.66) node [anchor=north west][inner sep=0.75pt]  [color={rgb, 255:red, 208; green, 2; blue, 27 }  ,opacity=1 ]  {$\left( g_{3} ,d_{0} g_{2} ,d_{1}^{2} g_{1}\right)$};
\draw (501,202.4) node [anchor=north west][inner sep=0.75pt]  [color={rgb, 255:red, 208; green, 2; blue, 27 }  ,opacity=1 ]  {$\left( g_{3} ,d_{0} g_{2} ,d_{0}^{2} g_{1}\right)$};
\draw    (96,210) -- (140,210) ;
\draw [shift={(142,210)}, rotate = 180] [color={rgb, 255:red, 0; green, 0; blue, 0 }  ][line width=0.75]    (10.93,-3.29) .. controls (6.95,-1.4) and (3.31,-0.3) .. (0,0) .. controls (3.31,0.3) and (6.95,1.4) .. (10.93,3.29)   ;
\draw    (173.34,198) -- (174.6,154) ;
\draw [shift={(174.66,152)}, rotate = 91.64] [color={rgb, 255:red, 0; green, 0; blue, 0 }  ][line width=0.75]    (10.93,-3.29) .. controls (6.95,-1.4) and (3.31,-0.3) .. (0,0) .. controls (3.31,0.3) and (6.95,1.4) .. (10.93,3.29)   ;
\draw    (56.66,198) -- (55.4,154) ;
\draw [shift={(55.34,152)}, rotate = 88.36] [color={rgb, 255:red, 0; green, 0; blue, 0 }  ][line width=0.75]    (10.93,-3.29) .. controls (6.95,-1.4) and (3.31,-0.3) .. (0,0) .. controls (3.31,0.3) and (6.95,1.4) .. (10.93,3.29)   ;
\draw    (86,140) -- (150,140) ;
\draw [shift={(152,140)}, rotate = 180] [color={rgb, 255:red, 0; green, 0; blue, 0 }  ][line width=0.75]    (10.93,-3.29) .. controls (6.95,-1.4) and (3.31,-0.3) .. (0,0) .. controls (3.31,0.3) and (6.95,1.4) .. (10.93,3.29)   ;
\draw    (77.23,198) -- (153.05,153.02) ;
\draw [shift={(154.77,152)}, rotate = 149.32] [color={rgb, 255:red, 0; green, 0; blue, 0 }  ][line width=0.75]    (10.93,-3.29) .. controls (6.95,-1.4) and (3.31,-0.3) .. (0,0) .. controls (3.31,0.3) and (6.95,1.4) .. (10.93,3.29)   ;
\draw [color={rgb, 255:red, 208; green, 2; blue, 27 }  ,draw opacity=1 ]   (274.22,185.95) -- (281.24,125.99) ;
\draw [shift={(281.48,124)}, rotate = 96.68] [color={rgb, 255:red, 208; green, 2; blue, 27 }  ,draw opacity=1 ][line width=0.75]    (10.93,-3.29) .. controls (6.95,-1.4) and (3.31,-0.3) .. (0,0) .. controls (3.31,0.3) and (6.95,1.4) .. (10.93,3.29)   ;
\draw    (318.56,220.95) -- (359.24,236.29) ;
\draw [shift={(361.11,237)}, rotate = 200.67] [color={rgb, 255:red, 0; green, 0; blue, 0 }  ][line width=0.75]    (10.93,-3.29) .. controls (6.95,-1.4) and (3.31,-0.3) .. (0,0) .. controls (3.31,0.3) and (6.95,1.4) .. (10.93,3.29)   ;
\draw    (294.28,124) -- (391.01,235.49) ;
\draw [shift={(392.32,237)}, rotate = 229.06] [color={rgb, 255:red, 0; green, 0; blue, 0 }  ][line width=0.75]    (10.93,-3.29) .. controls (6.95,-1.4) and (3.31,-0.3) .. (0,0) .. controls (3.31,0.3) and (6.95,1.4) .. (10.93,3.29)   ;
\draw    (348,102.11) -- (450.02,88.16) ;
\draw [shift={(452,87.89)}, rotate = 172.21] [color={rgb, 255:red, 0; green, 0; blue, 0 }  ][line width=0.75]    (10.93,-3.29) .. controls (6.95,-1.4) and (3.31,-0.3) .. (0,0) .. controls (3.31,0.3) and (6.95,1.4) .. (10.93,3.29)   ;
\draw [color={rgb, 255:red, 208; green, 2; blue, 27 }  ,draw opacity=1 ]   (331.67,187.42) -- (410.07,166.31) ;
\draw [shift={(412,165.79)}, rotate = 164.93] [color={rgb, 255:red, 208; green, 2; blue, 27 }  ,draw opacity=1 ][line width=0.75]    (10.93,-3.29) .. controls (6.95,-1.4) and (3.31,-0.3) .. (0,0) .. controls (3.31,0.3) and (6.95,1.4) .. (10.93,3.29)   ;
\draw    (482.75,132.26) -- (507.56,93.68) ;
\draw [shift={(508.64,92)}, rotate = 122.74] [color={rgb, 255:red, 0; green, 0; blue, 0 }  ][line width=0.75]    (10.93,-3.29) .. controls (6.95,-1.4) and (3.31,-0.3) .. (0,0) .. controls (3.31,0.3) and (6.95,1.4) .. (10.93,3.29)   ;
\draw    (520.86,92) -- (551.74,196.08) ;
\draw [shift={(552.31,198)}, rotate = 253.47] [color={rgb, 255:red, 0; green, 0; blue, 0 }  ][line width=0.75]    (10.93,-3.29) .. controls (6.95,-1.4) and (3.31,-0.3) .. (0,0) .. controls (3.31,0.3) and (6.95,1.4) .. (10.93,3.29)   ;
\draw [color={rgb, 255:red, 208; green, 2; blue, 27 }  ,draw opacity=1 ]   (494.39,167.26) -- (533.02,196.79) ;
\draw [shift={(534.61,198)}, rotate = 217.39] [color={rgb, 255:red, 208; green, 2; blue, 27 }  ,draw opacity=1 ][line width=0.75]    (10.93,-3.29) .. controls (6.95,-1.4) and (3.31,-0.3) .. (0,0) .. controls (3.31,0.3) and (6.95,1.4) .. (10.93,3.29)   ;
\draw    (467,239.03) -- (496.06,231.47) ;
\draw [shift={(498,230.97)}, rotate = 165.43] [color={rgb, 255:red, 0; green, 0; blue, 0 }  ][line width=0.75]    (10.93,-3.29) .. controls (6.95,-1.4) and (3.31,-0.3) .. (0,0) .. controls (3.31,0.3) and (6.95,1.4) .. (10.93,3.29)   ;
\draw [color={rgb, 255:red, 208; green, 2; blue, 27 }  ,draw opacity=1 ]   (331.67,205.96) -- (496,212.9) ;
\draw [shift={(498,212.99)}, rotate = 182.42] [color={rgb, 255:red, 208; green, 2; blue, 27 }  ,draw opacity=1 ][line width=0.75]    (10.93,-3.29) .. controls (6.95,-1.4) and (3.31,-0.3) .. (0,0) .. controls (3.31,0.3) and (6.95,1.4) .. (10.93,3.29)   ;
\draw [color={rgb, 255:red, 208; green, 2; blue, 27 }  ,draw opacity=1 ]   (317.15,124) -- (509.66,197.29) ;
\draw [shift={(511.53,198)}, rotate = 200.84] [color={rgb, 255:red, 208; green, 2; blue, 27 }  ,draw opacity=1 ][line width=0.75]    (10.93,-3.29) .. controls (6.95,-1.4) and (3.31,-0.3) .. (0,0) .. controls (3.31,0.3) and (6.95,1.4) .. (10.93,3.29)   ;
\draw    (306.6,185.95) -- (374.91,151.22)(384.71,146.24) -- (489.64,92.91) ;
\draw [shift={(491.42,92)}, rotate = 153.06] [color={rgb, 255:red, 0; green, 0; blue, 0 }  ][line width=0.75]    (10.93,-3.29) .. controls (6.95,-1.4) and (3.31,-0.3) .. (0,0) .. controls (3.31,0.3) and (6.95,1.4) .. (10.93,3.29)   ;

\end{tikzpicture}

\vspace{.5cm}

The diagram on the left is an illustration of the nondegenerate simplices in $\mathfrak{C}(\Delta^3)(0,3) \cong \Delta^1 \times \Delta^1$ and similarly on the right for $G(\Delta^3)(0,3) \cong \Delta^0 \times \Delta^1 \times \Delta^2$. The red subdiagram on the right shows the image of the Szczarba map.

\section{Simplicial models for the path category}

The goal of this section is to prove Theorem \ref{theorem1} from the introduction. In particular, we will prove that for any simplicial set there are natural weak equivalences of simplicial categories $$\widehat{\mathfrak{C}}(X) \longrightarrow \mathbb{C}(X) \xrightarrow{Sz_X} \mathbb{G}(X).$$ We then show that these three simplicial categories are connected via natural weak equivalences to the  the path category $\mathbb{P}(X)$. 

\subsection{Comparing the classifying space and homotopy coherent nerve functors} 
The main technical result in this section is the following.

\begin{Prop}\label{WandN}
For any simplicial groupoid $C \in \cat{Gpd}_{\cat{sSet}}$, the natural transformation $\mathfrak{C} \xRightarrow{\text{Sz}} G$ induces a weak homotopy equivalence of simplicial sets $\conj{W}\iota(C) \to \mathfrak{N}\iota(C)$.
\end{Prop}

\begin{proof}
Without loss of generality, we may assume $C$ has one object, i.e. $C$ is a a simplicial group. We will show that the natural transformation $\mathfrak{C} \xRightarrow{\text{Sz}} G$ induces an isomorphism of groups $\pi_n(\conj{W}\iota(C)) \to \pi_n(\mathfrak{N}\iota(C))$ for all $n \geq 1$.  Denote $S^n= \Delta^n / \partial\Delta^n \in \cat{sSet}^0$ and let $\pi: \Delta^n \to S^n$ denote the natural quotient map. 

We have natural isomorphisms of sets
\begin{equation}
\begin{aligned}
    \pi_n(\conj{W}\iota(C)) & \cong \cat{sSet}^0_{KQ}[S^n, \conj{W}^{Kan}(C)] \cong \cat{sGrp}_{KQ}[G^{Kan}(S^n), C] \\
    & \cong \cat{sGrp}_{KQ}[\mathcal{L}F^*(S^{n-1}), C] \cong \cat{sSet}^*_{KQ}[S^{n-1}, U^*\iota(C)],
    \end{aligned}
\end{equation}
which we now explain (for information on the notation see Appendix \ref{appendixmodel}). The first isomorphism follows from the definition of simplicial homotopy groups of Kan complexes. The second isomorphism follows from the Quillen equivalence $G^{Kan}: \cat{sSet}_{KQ}^0 \rightleftarrows \cat{sGrp}_{KQ} : \conj{W}^{Kan}$. The third isomorphism is obtained by noting that $G^{Kan}(S^n)$ is the free simplicial group generated by the simplicial set $S^{n-1}$. In fact, the isomorphism $G^{Kan}(S^n) \to \mathcal{L}F(S^{n-1})$ is determined by $\conj{\iota_n} \mapsto \iota_{n-1}$. The last isomorphism follows from Proposition \ref{locfree}. 

Similarly, we have natural isomorphisms of sets
\begin{equation}
    \begin{aligned}
    \pi_n(\mathfrak{N}\iota(C)) & \cong \cat{sSet}^0_{KQ}[S^n, \mathfrak{N} \iota(C)] \cong \cat{sSet}_{KQ}[S^n, \mathfrak{N}\iota(C)] \cong \cat{sSet}_J[S^n, \mathfrak{N} \iota(C)]   \\
    & \cong \cat{sGrp}_K[\mathcal{L} \mathfrak{C}(S^n), C] \cong \cat{sGrp}_K[\mathcal{L}F^* \left( (S^1)^{\wedge (n-1)} \right), C)] \cong \cat{sSet}^*_{KQ}[(S^1)^{\wedge (n-1)}, U^*\iota(C)],
    \end{aligned}
\end{equation}
which we now explain. The first isomorphism follows from the definition of simplicial homotopy groups of Kan complexes. The second isomorphism follows since $\cat{sSet}^0_{KQ} \hookrightarrow \cat{sSet}_{KQ}$ induces a fully faithful functor on homotopy categories \cite[Chapter V, Remark 6.5]{goerss2009simplicial}. The third isomorphism follows because $\mathfrak{N} \iota(C)$ is a Kan complex. The fourth isomorphism follows from the composition of  Quillen adjunctions $$\cat{sSet}_J \overset{\mathfrak{C}}{\underset{\mathfrak{N}}{\rightleftarrows}} \cat{Cat}_{\cat{sSet}, B} \overset{\mathcal{L}}{\underset{\iota}{\rightleftarrows}}  \cat{Gpd}_{\cat{sSet}, DK}.$$ The fifth isomorphism is given by the description of the simplicial sets of morphisms of $\mathfrak{C}(S^n)$ in terms of necklaces, namely,  $$\mathfrak{C}(S^n)(*,*)= \underset{(f: N \to S^n) \in ( \cat{Nec} \downarrow S^n)}{\text{colim}}(\Delta^1)^{\times \text{dim}(N)} \cong \underset{(p: \Delta^n \vee ... \vee \Delta^n \to S^n) \in (\cat{Nec} \downarrow S^n)}{\text{colim}}(\Delta^1)^{\times k(n-1)},$$ where the second colimit of simplicial sets is taken over the full subcategory of $\cat{Nec} \downarrow S^n$ on the set of objects  $$\{ (p: N \to S^n)| N= \Delta^{n_1} \vee ... \vee \Delta^{n_k}, \text{for some $k\geq 1$}, n_i=n \text{ for all $i$}, \text{and } p=\pi \vee ... \vee \pi\}.$$ From this description of $\mathfrak{C}(S^n)$, it follows that  $\mathfrak{C}(S^n) \cong F \left( (S^1)^{\wedge (n-1)} \right),$ so $\mathfrak{C}(S^n)$ is freely generated by the $(n-1)!$ non-degenerate $(n-1)$-dimensional simplices of $(S^1)^{\wedge (n-1)}$. 
The last isomorphism follows from Proposition \ref{locfree}. 

Under the identifications above, the map $\pi_n(\conj{W}\iota(C)) \to \pi_n(\mathfrak{N}\iota(C))$ becomes a map 
$$\cat{sSet}^*_{KQ}[S^{n-1}, U^*\iota(C)] \to \cat{sSet}^*_{KQ}[(S^1)^{\wedge (n-1)}, U^*\iota(C)].$$ We claim this is an isomorphism. In fact, this follows by considering the diagram
\begin{equation}
    \begin{tikzcd}
	{\mathfrak{C}(\Delta^ n)} && {G(\Delta^n)} \\
	{ \mathfrak{C}(S^n)} && {G(S^n)}
	\arrow["{  \text{Sz}_{\Delta^n}}", from=1-1, to=1-3]
	\arrow["{ \text{Sz}_{S^n}}"', from=2-1, to=2-3]
	\arrow["{\mathfrak{C}(\pi)}", from=1-1, to=2-1]
	\arrow["{G(\pi)}",from=1-3, to=2-3]
\end{tikzcd}
\end{equation}
and noting that $F^*( (S^1)^{\wedge (n-1} ) \cong \mathfrak{C}(S^n) \xrightarrow{\text{Sz}_{S^n}} G(S^n) \cong F^*(S^{n-1}) $ is induced by the map $(S^1)^{\wedge (n-1)} \to S^{n-1}$ which collapses all nondegenerate $(n-1)$-simplices except for the one labelled $\gamma(\sigma) = (0, \dots, 0)$, that is sent to $\iota_{n-1}$ in $S^{n-1}$. This map is a weak equivalence and consequently it induces an isomorphism $\cat{sSet}^*_{KQ}[S^{n-1}, U^*\iota (C)] \to \cat{sSet}^*_{KQ}[(S^1)^{\wedge (n-1)}, U^*\iota(C)]$, as desired. 
\end{proof}
\subsection{Comparing the localized rigidification and Kan loop groupoid functors}
We can now deduce the first part of Theorem \ref{theorem1}. 

\begin{Cor}\label{LandG}For any simplicial set $X$ the map $\mathcal{L} \text{Sz}_X: \mathcal{L} \mathfrak{C}(X) \to G^{Kan}(X)$ is a weak equivalence of simplicial groupoids. 
\end{Cor}
\begin{proof}

By Corollary \ref{Qequi} $$\mathcal{L}\mathfrak{C}: \mathsf{sSet}_{KQ} \rightleftarrows \mathsf{Gpd}_{\mathsf{sSet}, DK}: \mathfrak{N} \iota$$ defines a Quillen equivalence. By Proposition \ref{WandN}, the natural transformation $\mathcal{L}Sz: \mathcal{L}\mathfrak{C} \xRightarrow{} G^{Kan}$ induces a weak equivalence between their right Quillen adjoints. Hence, $\mathcal{L}Sz_X: \mathcal{L}\mathfrak{C}(X) \xRightarrow{ } G^{Kan}(X)$ is a natural weak equivalence of simplicial groupoids for all $X$. 
\end{proof}

Fix a fibrant replacement functor $\mathcal{J}: \cat{sSet} \to \cat{sSet}$ in the Joyal model structure. For any simplicial set $X$ denote by $X^1$ its $1$-skeleton. Fix a simplicial set $J^1$ such that there is a cofibration $\Delta^1 \to  J^1$ which is a weak homotopy equivalence and such that the homotopy category of $J^1$ is a groupoid.  Then for any simplicial set $X$ there is a natural cofibration 
$$X^1 \to \underset{\Delta^1 \to X^1}{ \text{colim}} J^1.$$
The above cofibration is a weak homotopy equivalence which we may think of as a thickening of $X^1$ in which every $1$-simplex has an inverse up to homotopy.

Denote the pushout of the cofibrations $X^1 \to X$ and $X^1 \to \underset{\Delta^1 \to X^1}{ \text{colim}} J^1$ by $$\mathcal{K}_1(X)= \left( \underset{\Delta^1 \to X^1}{ \text{colim}} J^1\right) \coprod_{X^1} X.$$ Once we fix $J^1$ and a trivial cofibration (in the Kan-Quillen model structure) $\Delta^1 \to J^1$, this construction gives rise to a functor $$\mathcal{K}_1: \cat{sSet} \to \cat{sSet}.$$

 A fundamental principle of the theory of quasi-categories is that fibrant replacement of a simplicial set $X$ in the Kan-Quillen model structure can be thought of as first "inverting" the $1$-skeleton of $X$ up to homotopy and then fibrant replacing the resulting simplicial set in the Joyal model structure.

\begin{Prop} \label{Kanrep} The composition of natural maps of simplicial sets $$X \to \mathcal{K}_1(X) \to \mathcal{J} (\mathcal{K}_1(X))$$ is a cofibration and a weak homotopy equivalence. Moreover, $\mathcal{J} (\mathcal{K}_1(X))$ is fibrant in the Kan-Quillen model structure, i.e. $\mathcal{J} (\mathcal{K}_1(X))$ is a Kan complex.
\end{Prop}
\begin{proof}

Since $X = X^1 \amalg_{X^1} X$ is a pushout of cofibrations between cofibrant objects, this presents a homotopy pushout in the Kan-Quillen model structure. Hence the weak homotopy equivalence $X^1 \to \underset{\Delta^1 \to X^1}{ \text{colim}} J^1$ induces a weak homotopy equivalence $X \to \mathcal{K}_1(X)$. Since every Joyal equivalence is a weak homotopy equivalence it follows that $\mathcal{K}_1(X) \to \mathcal{J} (\mathcal{K}_1(X))$ is a weak homotopy equivalence. 
By Theorem \ref{Kaniff}, the quasi-category $\mathcal{J} (\mathcal{K}_1(X))$ is a Kan complex since the homotopy category, which is isomorphic to the homotopy category of $\mathcal{K}_1(X)$, is a groupoid.
\end{proof}

Denote $\mathcal{K}(X)=\mathcal{J}(\mathcal{K}_1(X))$ so that $X \mapsto \mathcal{K}(X)$ is a functorial Kan replacement. 

\begin{Cor}\label{K1andK} The simplicial categories $\mathfrak{C}(\mathcal{K}_1(X))$ and $\mathfrak{C}(\mathcal{K}(X))$ are naturally Bergner weak equivalent.
\end{Cor}
\begin{proof}

Since  $\mathcal{K}_1(X) \xrightarrow{\simeq} \mathcal{J}(\mathcal{K}_1(X))=\mathcal{K}(X)$ is a Joyal equivalence and $\mathfrak{C}$ sends Joyal equivalences to Bergner weak equivalences of simplicial categories the result follows.
\end{proof}

For simplicity from now on we will denote $$\mathbb{C}= \iota \circ \mathcal{L} \circ \mathfrak{C}: \cat{sSet} \to \cat{Cat}_{\cat{sSet}}$$
and
$$\mathbb{G}= \iota \circ \mathcal{L} \circ G: \cat{sSet} \to \cat{Cat}_{\cat{sSet}}.$$

 We now argue that $\mathfrak{C} (\mathcal{K}(X))$ can be modeled, up to weak equivalence of simplicial categories, by localizing the image of $\mathfrak{C}(X^1)$ in $\mathfrak{C}(X)$. Define $\widehat{\mathfrak{C}}: \cat{sSet} \to \cat{Cat}_{\cat{sSet}}$ to be the functor given by $$\widehat{\mathfrak{C}}(X)= \mathfrak{C}(X)[\mathfrak{C}(X^1)^{-1}].$$
\begin{Prop}\label{L1andL} For any $X \in \cat{sSet}$, the natural inclusion of simplicial categories $\widehat{\mathfrak{C}}(X) \to \mathbb{C}(X)$ is a weak equivalence. 
\end{Prop}
\begin{proof} This follows from Propositions 9.5 and 9.6 of \cite{dwyer1980simplicial}. 
\end{proof} 

\begin{Cor}\label{corthm1} The natural transformation $Sz: \mathfrak{C} \xRightarrow{ } G$ induces a weak equivalence of simplicial categories $\iota(\mathcal{L} (Sz_{X})): \widehat{\mathfrak{C}}(X) \to \mathbb{G}(X)$. 
\end{Cor}
\begin{proof} This follows from Corollary \ref{LandG} and Proposition \ref{L1andL}.
\end{proof}

\begin{Prop} \label{inverting1} The simplicial categories $\widehat{\mathfrak{C}}(X)$ and $\mathfrak{C} (\mathcal{K}_1(X))$  are naturally weak equivalent.
\end{Prop}
\begin{proof} Since $\mathfrak{C}$ is a left adjoint, it preserves pushouts so we have a natural isomorphism
$$\mathfrak{C}(\mathcal{K}_1(X)) \cong \mathfrak{C}(\mathcal{K}_1(X^1)) \coprod_{\mathfrak{C}(X^1)} \mathfrak{C}(X).$$
We claim that the natural maps of simplicial categories $$\mathfrak{C}(\mathcal{K}_1(X^1)) \xrightarrow{ } \mathbb{C}(\mathcal{K}_1(X^1)) \xleftarrow{ } \mathbb{C} (X^1)$$
are all weak equivalences. Since $\mathfrak{C}(\mathcal{K}_1(X^1)$ is a cofibrant simplicial category with the property that the homotopy category is a groupoid, Proposition 9.5 of \cite{dwyer1980simplicial} implies that the first map is a weak equivalence of simplicial categories. 

By Corollary \ref{Qequi}, $\mathcal{L} \circ \mathfrak{C}: \cat{sSet}_{KQ} \to \cat{Gpd}_{\cat{sSet}, DK}$ preserves weak equivalences between cofibrant objects so $\mathcal{L}\mathfrak{C}(X^1) \to \mathcal{L} \mathfrak{C}(\mathcal{K}_1(X^1))$ is a weak equivalence. Since $\iota: \cat{Gpd}_{\cat{sSet}, DK} \to \cat{Cat}_{\cat{sSet},B}$ preserves weak equivalences, it follows that $\mathbb{C}(X^1) \to \mathbb{C}(\mathcal{K}_1(X^1))$ is a weak equivalence of simplicial categories. Hence, we have natural induced weak equivalences of simplicial categories

$$ \mathfrak{C}(\mathcal{K}_1(X)) \cong \mathfrak{C}(\mathcal{K}_1(X^1)) \coprod_{\mathfrak{C}(X^1)} \mathfrak{C}(X) \xrightarrow{\simeq} \mathbb{C}(\mathcal{K}_1(X^1)) \coprod_{\mathfrak{C}(X^1)} \mathfrak{C}(X) \xleftarrow{\simeq} \mathbb{C}(X^1) \coprod_{\mathfrak{C}(X^1)} \mathfrak{C}(X) \cong \widehat{\mathfrak{C}}(X).$$

\end{proof}


\begin{Cor} Let $f: X \to Y$ be a map of simplicial sets. The following are equivalent:
\begin{enumerate}
    \item $f: X \to Y$ is a weak homotopy equivalence
    \item $\widehat{\mathfrak{C}}(f): \widehat{\mathfrak{C}}(X) \to \widehat{\mathfrak{C}}(Y)$ is a weak equivalence of simplicial categories
    \item $\mathcal{L}\mathfrak{C}(f): \mathcal{L}\mathfrak{C}(X) \to\mathcal{L} \mathfrak{C}(Y)$ is a weak equivalence of simplicial groupoids.
\end{enumerate}
\end{Cor}
\begin{proof} The equivalence between (2) and (3) follows from Proposition \ref{L1andL}. The equivalence between (1) and (3) follows from Corollary \ref{Qequi} since every object of $\cat{sSet}_{KQ}$ is cofibrant and every object of $\cat{Gpd}_{\cat{sSet}, DK}$ is fibrant. 
\end{proof}

\subsection{Comparing the localized rigidification and the path category functors}
Given any simplicial set $X$, denote by $\mathbb{P}(X) = \text{Sing} \mathcal{P} |X|$ the simplicial category defined by applying the functor $\text{Sing}: \cat{Top} \to \cat{sSet}$ to the morphisms spaces of $\mathcal{P}|X|$. Let $\mathcal{N}: \cat{Cat}_{\cat{sSet}} \to \cat{sSet}$ be the classical simplicial nerve functor defined by applying the ordinary nerve level-wise to obtain a bisimplicial set and then taking the diagonal. 

\begin{Prop} The simplicial categories $\mathbb{P}(X)$ and $\mathbb{C}(X)$ are naturally weak equivalent. 
\end{Prop} 

\begin{proof} 
Let $\mathcal{C}$ be a fibrant groupoid, i.e. a fibrant simplicial category such that the homotopy category is a groupoid. We claim that for any such $\mathcal{C}$ there are natural weak homotopy equivalences of Kan complexes
\begin{equation}\label{weakeqs}
\mathcal{N}(\mathcal{C} ) \xrightarrow{\simeq} \overline{W}(\mathcal{C}) \xrightarrow{\simeq} \mathfrak{N}(\mathcal{C}).
\end{equation}
Note that any such $\mathcal{C}$ is naturally weak equivalent to a simplicial groupoid. In fact, let $\mathcal{Q}: \cat{Cat}_{\cat{sSet}} \to \cat{Cat}_{\cat{sSet}}$ be a cofibrant replacement functor in the Bergner model structure and note there are natural weak equivalences of fibrant simplicial categories $$\mathcal{C} \xleftarrow{\simeq } \mathcal{Q}(\mathcal{C}) \xrightarrow{ \simeq}  \iota \mathcal{L}\mathcal{Q}(\mathcal{C}).$$ The first map is a weak equivalence by definition and the second by Proposition 9.5 of \cite{dwyer1980simplicial}. The three functors $\mathcal{N}$, $\conj{W},$ and $\mathfrak{N}$ send weak equivalences between fibrant groupoids to weak homotopy equivalences between Kan complexes, hence we may assume (without loss of generality) that $\mathcal{C}$ is a simplicial groupoid.

When $\mathcal{C}$ is a simplicial groupoid, the weak homotopy equivalence $\conj{W}(\mathcal{C}) \to \mathfrak{N}(\mathcal{C})$ is given by the map $Sz$, as proven in Proposition \ref{WandN}. A natural weak homotopy equivalence $\mathcal{N}(C) \to \conj{W}(\mathcal{C})$ is constructed in A.5.1 of \cite{hinich2001dg}. 

We now argue that $\mathcal{N}(\mathbb{P}(X) )$ and $X$ are naturally weak homotopy equivalent. Without loss of generality we may assume $X$ is connected. Then the space $|\mathcal{N}\mathbb{P}(X)|$ is weak homotopy equivalent to $\mathcal{B} \Omega|X|$, where $\mathcal{B}$ is the classifying space functor (or geometric bar construction) constructed in \cite{may1975classifying} and $\Omega|X|$ is the (Moore) based loop space of $|X|.$ By Lemma 15.4 of \cite{may1975classifying}, there is a natural weak homotopy equivalence $\mathcal{B}\Omega |X| \to |X|.$ Hence, $\mathcal{N}(\mathbb{P}(X) )$ and $X$ are naturally weak homotopy equivalent. Consequently, $\mathfrak{N}(\mathbb{P}(X) ) $ and $X$ are naturally weak homotopy equivalent. By Corollary \ref{Qequi}, $\mathbb{C}(X)$ and $\mathbb{P}(X)$ are naturally weak equivalent as simplicial categories. 

\end{proof}

We summarize all of the weak equivalences between the simplicial categories described above. These are all combinatorial models for the path category of the geometric realization of an arbitrary simplicial set, thus proving the second part of Theorem \ref{theorem1}. 
\begin{Thm}\label{equivalencesofsimplicialcategories} For any simplicial set $X$, the following simplicial categories are all naturally weakly equivalent $$\mathbb{P}(X) \simeq \widehat{\mathfrak{C}}(X) \simeq  \mathbb{C}(X)\simeq  \mathfrak{C}(\mathcal{K}_1(X))\simeq \mathfrak{C}(\mathcal{K}(X)) \simeq \mathbb{G}(X).$$
\end{Thm}

\section{Algebraic models for the path category}
The goal of this section is to prove Theorem \ref{theorem2} in the introduction. The main construction, the functor denoted by $\widehat{\mathbf{\Omega}}$, takes a graded coalgebra equipped with extra structure and produces a category enriched over the monoidal category of differential graded (dg) coalgebras. When $\widehat{\mathbf{\Omega}}$ is applied to a suitable model for the chains on a simplicial set, it produces a model for its path category. Throughout this section we fix a commutative ring $R$. We often consider $R$ as a dg $R$-module concentrated in degree $0$. 

\subsection{Categorical coalgebras}
We describe a notion of coalgebras equipped with further structure to which one may naturally associate a dg category through a version of the cobar construction. Categorical coalgebras are particular examples of \textit{pointed curved coalgebras}, a notion introduced in \cite{holstein-lazarev}. Similar notions are studied in \cite{KM21}, in the setting of no differential and no curvature considerations.

\begin{Def} \label{Coalgebras}
A \textbf{categorical $R$-coalgebra} consists of the data $C=(C, \Delta, \partial, h)$
where
\begin{itemize}
    \item $C= \bigoplus_{i=0}^{\infty} C_i$  is a non-negatively graded $R$-module.
    \item $\Delta: C \to C \otimes C$ is a degree $0$ coassociative counital coproduct with counit $\varepsilon: C \to R$.

    \item 
    The set \[P_C:= \{ p \in C : \Delta(p)= p\otimes p, \varepsilon(p)=1_{R}\}\] of
``set-like'' elements in $C$ is non-empty and \[C_0 \cong R[P_C].\] 
We call the elements of $P_C$ the \textbf{objects} of $C$. 
    \item $\partial: C \to C$ is a linear map of degree $-1$ which is a graded coderivation of $\Delta.$
    \item The projection map $\epsilon: C \to C_0$ satisfies $\epsilon \circ \partial=0.$
    \item $h: C \to R$ is a linear map of degree $-2$ satisfying $h \circ \partial =0$ and $$\partial \circ \partial= (h \otimes \text{id}) \circ (\Delta - \Delta^{op})$$
    where $\Delta^{op}= t \circ \Delta$ for $t(x \otimes y)= (-1)^{|x||y|} y \otimes x$. The right hand side of the above equation is being considered as a map $C \to R \otimes C \cong C.$ The map $h$ is called the \textbf{curvature} of $C$. The above equation may be written as
    \[d^2(x) = \sum_{(x)} h(x')x'' + x'h(x'').\]

\end{itemize}
\end{Def}

\begin{Rem} Note that any categorical coalgebra $C$ has a natural $C_0$-bicomodule structure given by the maps
$$\rho_l: C \xrightarrow{\Delta} C \otimes C \xrightarrow{\epsilon \otimes \text{id}_C} C_0 \otimes C$$
and
$$\rho_r: C \xrightarrow{\Delta} C \otimes C \xrightarrow{\text{id}_C \otimes \epsilon} C \otimes C_0.$$
Furthermore, the coassociativity of $\Delta: C \to C\otimes C$ implies that $\Delta(C) \subseteq C \underset{C_0}{\square} C,$
where 
$$C \underset{C_0}{\square} C := \text{ker} (\rho_r \otimes \text{id}_C - \text{id}_C \otimes \rho_l) \subseteq C \otimes C.$$ 
More generally, for any dg $R$-coalgebra $C$, any right dg $C$-comodule $M$, and any left dg $C$-comodule $N$, define $$M \underset{C}{\square} N = \text{ker} (\rho_r \otimes \text{id}_N - \text{id}_M \otimes \rho_l) \subseteq M \otimes N,$$ where $\rho_r: M \to M \otimes C$ and $\rho_l: N \to C \otimes N$ are the respective coactions. This is called the \textbf{cotensor product of $M$ and $N$ over $C$}. When there is no risk of confusion we write $M \square N.$
\end{Rem}

\begin{Rem}
When $R$ is a field, categorical coalgebras are special types of \textit{pointed curved coalgebras} as defined in \cite{holstein-lazarev}. A pointed curved coalgebra is a curved coalgebra whose coradical is a direct sum of copies of $R$ together with the structure of a splitting of the coradical satisfying certain properties. In the case of categorical coalgebras, we have a unique splitting given by the projection map $\epsilon: C \to C_0.$ 

\end{Rem}

Morphisms of categorical coalgebras are defined as follows.

\begin{Def} A \textbf{morphism between categorical coalgebras} $C=(C, \Delta, \partial, h)$ and  $C'=(C', \Delta', \partial', h')$  consists of a pair $(f_0, f_1)$ where

\begin{itemize}
\item $f_0: (C, \Delta) \to (C', \Delta')$ is a morphism of graded $R$-coalgebras, 
\item $f_1: C \to C'_0$ a $C_0'$-bicomodule map such that the composition $\bar{f_1}= \varepsilon' \circ f_1$, where $\varepsilon'$ is the counit of $C'$, is a degree $-1$ map satisfying
$$f_0 \circ \partial = \partial' \circ f_0 + (\bar{f_1} \otimes f_0) \circ (\Delta - \Delta^{op})$$
and
$$h' \circ f_0= h + \bar{f_1} \circ \partial + (\bar{f_1} \otimes \bar{f_1}) \circ \Delta.$$ 
\end{itemize}
The composition of two morphisms of categorical coalgebras is defined by  $$(g_0, g_1) \circ (f_0, f_1)= (g_0 \circ f_0 , g_1 \circ f_0 + g_0 \circ f_1).$$
Denote by $\cat{cCoalg}_R$ the category of categorical coalgebras. 
\end{Def}

\begin{Def} Denote by $\cat{dgCat}_R^{\geq 0}$ the category of non-negatively graded dg categories, i.e. categories enriched over the monoidal category $(\cat{Ch}_R^{\geq 0}, \otimes)$ of non-negatively graded dg $R$-modules with tensor product. For short, we will call the objects in $\cat{dgCat}_R^{\geq 0}$ \textbf{dg categories}. Given any dg category $\cat{C}$ denote by $H_0(\cat{C})$ the category enriched over $R$-modules obtained by applying the $0$-th homology functor at the level of morphisms. More precisely, $H_0(\cat{C})$ has the same objects as $\cat{C}$ and for any two objects $x, y \in \cat{C}$, $H_0(\cat{C})(x,y)$ is the $0$-th homology of the dg $R$-module $\cat{C}(x,y).$ The composition in $H_0(\cat{C})$ is induced by the composition in $\cat{C}.$
A morphism of dg categories $f: \cat{C} \to \cat{D}$ is called a \textbf{quasi-equivalence} if it induces an equivalence of categories $H_0(f): H_0(\cat{C}) \xrightarrow{\cong} H_0(\cat{D})$ and for any $x,y \in \cat{C}$ the induced map $f: \cat{C}(x,y) \to \cat{D}(f(x),f(y))$ is a quasi-isomorphism. 
\end{Def}

Any categorical coalgebra gives rise to a dg category through the following version of the cobar construction, which is a many object generalization of the classical cobar construction for coaugmented coalgebras. 

\begin{Def}
Define the \textbf{cobar functor} $$\Omega: \cat{cCoalg}_R \to \cat{dgCat}_R^{\geq 0}$$ as follows. Given any $C= (C, \Delta, \partial, h) \in \cat{cCoalg}_R$, the objects of $\Omega(C)$ are the elements of the set $P_C$.

Any element $x \in P_C$ determines a map $i_x: R\to C_0=R[P_C]$ given by $i_x(1_R)= x$. The map $i_x$ gives rise to a $C_0$-bicomodule structure on $R$ with $C_0$-coaction maps $$R \cong R \otimes R \xrightarrow{i_x \otimes \text{id}_R} C_0 \otimes R$$ and $$R \cong R \otimes R \xrightarrow{\text{id}_R \otimes i_x} R \otimes C_0.$$ We denote this $C_0$-bicomodule by $R[x].$

Write $C= \bar{C} \oplus C_0$ and denote by $s^{-1}\bar{C}$ the graded $R$-module obtained by applying the shift by $-1$ functor to the graded $R$-module $\bar{C}$. We have degree $-1$ induced maps $\bar{\partial}: s^{-1}\bar{C} \to s^{-1}\bar{C}$, $\bar{\Delta}: s^{-1}\bar{C} \to s^{-1}\bar{C} \otimes s^{-1}\bar{C}$, and $\bar{h}: s^{-1}\bar{C} \to C_0$, where $\bar{h}$ is defined as 
the composition $$s^{-1}\bar{C} \xrightarrow{s^{+1}} \bar{C} \xrightarrow{\rho_r} C \otimes C_0 \xrightarrow{h \otimes \textit{id}} R \otimes C_0 \cong C_0.$$ 

For any two $x,y \in P_C$ define a non-negatively graded $R$-module by
$$\Omega(C)(x,y)= \bigoplus_{n=0}^{\infty} R[x] \square (s^{-1}\bar{C})^{\square n} \square R[y],$$
where $(s^{-1}\bar{C})^{\square n} $ denotes the $n$-fold cotensor product of the $C_0$-bicomodule $s^{-1}\bar{C}$ with itself. We define $(s^{-1}\bar{C})^{\square 0}=C_0$. 
The differential $$D: \Omega(C)(x,y)_n \to \Omega(C)(x,y)_{n-1}$$ is determined by 
$$D = \bar{h} + \bar{\partial} + \bar{\Delta}.$$

An straightforward computation yields $D \circ D=0,$ see Lemma 3.18 in \cite{holstein-lazarev}. The composition in $\Omega(C)$ is given by concatenation of monomials. This construction is functorial with respect to morphisms of categorical coalgebras, see Lemma 3.19 in \cite{holstein-lazarev}.
\end{Def}

\begin{Rem} The cobar functor defined above is part of a more general construction connecting pointed curved coalgebras and dg categories \cite{holstein-lazarev}. When $R$ is a field, this construction gives rise to a Quillen equivalence of model categories extending the corresponding result for conilpotent curved coalgebras and dg algebras. 
\end{Rem}

\subsection{Normalized chains as a categorical coalgebra}

Denote by $(\cat{dgCoalg}_R^{\geq 0}, \otimes)$ the monoidal category of non-negatively graded dg coassociative counital $R$-coalgebras. For short, we will call the objects of $\cat{dgCoalg}_R^{\geq 0}$ dg coalgebras.

Given any simplicial set $X$, denote by  $(N^{\Delta}_*(X), \partial)$ the dg $R$-module of normalized simplicial chains. The Alexander-Whitney coproduct, given on any simplex $\sigma \in X_n$ by
$$\Delta(\sigma) = \sum_{i=0}^n \sigma(0,...,i) \otimes \sigma(i,...,n),$$
induces a coassociative coproduct $$\Delta: N^{\Delta}_*(X) \to N^{\Delta}_*(X) \otimes N^{\Delta}_*(X)$$ of degree $0$. In the above formula, $\sigma(0,...,i)$ and $\sigma(i,...,n)$ denote the first $i$-th and last $(n-i)$-th faces of $\sigma,$ respectively. 
This construction gives rise to a functor 
$$C^{\Delta}_*: \cat{sSet} \to \cat{dgCoalg}^{\geq 0}_R$$ given by 
$$C^{\Delta}_*(X)=(N^{\Delta}_*(X), \partial, \Delta).$$

For any two simplicial sets $X$ and $Y$, the natural Eilenberg-Zilber shuffle map
$$\text{EZ}_{X,Y}: N^{\Delta}_*(X) \otimes N^{\Delta}_*(Y) \to N^{\Delta}_*(X \times Y)$$
is a map of dg coalgebras and consequently makes $C^{\Delta}_*$ into a lax monoidal functor, see 17.6 in \cite{EilMoo66}.

The projection map $\epsilon: N^{\Delta}_*(X) \to N^{\Delta}_0(X)$ does not satisfy $\epsilon \circ \partial=0.$ However, the differential $\partial$ may be modified to obtain a categorical coalgebra as follows

\begin{Def} \label{chains} Given any simplicial set $X \in \cat{sSet}$ define a categorical coalgebra $\widetilde{C}^{\Delta}_*(X) \in \cat{cCoalg_R}$ as follows.
The underlying graded $R$-module of  $\widetilde{C}^{\Delta}_*(X)$ is just the normalized chains $N^{\Delta}_*(X)$ given by $N^{\Delta}_n(X)= R[X_n]/D(X_n)$, where $D(X_n)\subseteq R[X_n]$ is the sub $R$-module  generated by degenerate $n$-simplices. 

Let $e: R[X_1] \to R$ be the linear map sending degenerate $1$-simplices to $0 \in R$ and non-degenerate $1$-simplices to $1 \in R$. The map $e$ induces a linear map $\widetilde{e}: N^{\Delta}_1(X) \to R.$ Define a new differential
$$\widetilde{\partial}: N^{\Delta}_*(X) \to N^{\Delta}_{*-1}(X)$$
by $$\widetilde{\partial}= \partial - (\text{id}\otimes \widetilde{e} - \widetilde{e}\otimes\text{id}) \circ \Delta.$$

The map $\widetilde{\partial}$ is a coderivation of $\Delta$ and the projection map $\epsilon: N^{\Delta}_*(X) \to N^{\Delta}_0(X)$ now satisfies $\epsilon \circ \widetilde{\partial}=0$. Finally, define $h: N^{\Delta}_2(X) \to R$ by $$h= (\widetilde{e} \otimes \widetilde{e}) \circ \Delta + \widetilde{e} \circ \partial$$

A routine check yields that
$$\widetilde{C}^{\Delta}_*(X)=(N^{\Delta}_*(X), \widetilde{\partial}, \Delta, h)$$
defines an object in $\cat{cCoalg}_R$. Furthermore, this construction gives rise to a functor
$$\widetilde{C}^{\Delta}_*: \cat{sSet} \to \cat{cCoalg}_R.$$

\end{Def}

\subsection{The cobar functor, necklaces, and cubes}
The cobar functor from categorical coalgebras to dg categories formalizes algebraically a combinatorial construction that associates to any simplicial set $X$ a higher category $\Pi(X)$. The objects of $\Pi(X)$ are the vertices of $X$ and the mapping spaces $\Pi(X)(x,y)$ are described in terms of necklaces inside $X$ starting at $x$ and ending at $y$. In this section, we define the functor $\Pi$, which may be regarded as a cubical version of $\mathfrak{C}$, and explain how it relates to the cobar functor. We build upon the notions introduced in section \ref{necklaces}.

\begin{Def} 
A \textbf{necklical set} is a functor $\cat{Nec}^{op} \to \cat{Set}$, i.e. a presheaf of sets over the category of necklaces. 
Let $\cat{nSet}$ be the category whose objects are necklical sets and morphisms are natural transformations. For example, any necklace $N \in \cat{Nec}$ gives rise to a necklical set $\mathcal{Y}(N)=\text{Hom}_{\cat{Nec}}( \_ , N)$, defining a Yoneda embedding functor
$$\mathcal{Y}: \cat{Nec} \to \cat{nSet}.$$ The monoidal structure on $\cat{Nec}$ induces a (non-symmetric) monoidal structure on $\cat{nSet}$ which we also denote by $$\vee: \cat{nSet} \times \cat{nSet} \to \cat{nSet}.$$ A small category enriched over the monoidal category $(\cat{nSet}, \vee)$ is called a \textbf{necklical category}. Denote by $\cat{Cat}_{\cat{nSet}}$ the category of necklical categories.
\end{Def}

\begin{Cons}
We define a functor $$\Pi: \cat{sSet} \to \cat{Cat}_{\cat{nSet}}$$ as follows. Given any $X \in \cat{sSet}$, the set objects of $\Pi(X)$ is defined to be $X_0$, the vertices of $X$. Given any two $x,y \in X_0$  define a necklical set of morphisms
$$\Pi(X)(x,y)= \underset{(f: N \to X) \in (\cat{Nec} \downarrow X)_{x,y}}{\text{colim}} \mathcal{Y}(N).$$ Composition is then induced by the monoidal structure $\vee: \cat{Nec} \times \cat{Nec} \to \cat{Nec}.$ This construction is clearly functorial with respect to maps of simplicial sets. 
\end{Cons}

The idea behind the construction is very natural: the mapping spaces $\Pi(X)(x,y)$ are obtained by gluing ``cells'' $\mathcal{Y}(N) \in \cat{nSet}$ corresponding to neckalces $N$ in $X$ starting at $x$ and ending at $y$. Composition is then given by concatenating necklaces. Each $\mathcal{Y}(N)$ may be interpreted as a cube as we now explain using the framework of cubical sets with connections. 

\begin{Def} Let $[1]=\{0 \to 1\}$ be the category with two objects and a single non-identity morphism between them. Denote by $[1]^n$ the Cartesian product of $n$ copies of $[1].$ We write $[1]^0$ for the category with one object and one morphism. Define a category $\square$ whose objects are the categories $\{ [1]^0, [1]^1, [1]^2,...\}$ and morphisms are generated by the following three types of functors:
\begin{itemize}
\item \textit{cubical co-face functors} $\delta^{\epsilon}_{j,n}: [1]^n \to [1]^{n+1}$, where $j=0,1,...,n+1$, and $\epsilon \in \{0,1\}$, defined by
\begin{eqnarray*}
\delta^{\epsilon}_{j,n}(t_1,...,t_n)=(t_1,...,t_{j-1},\epsilon,t_j,...,t_n),
\end{eqnarray*}
\item \textit{cubical co-degeneracy functors} $\varepsilon_{j,n}: [1]^n \to [1]^{n-1}$, where $j=1,...,n$, defined by
\begin{eqnarray*}
\varepsilon_{j,n}(t_1,...,t_n)=(t_1,...,t_{j-1},t_{j+1},...,t_n), \text{ and }
\end{eqnarray*}
\item \textit{cubical co-connection functors} $\gamma_{j,n}: [1]^n \to [1]^{n-1}$, where $j=1,...,n-1$, $n\geq 2$, defined by
\begin{eqnarray*}
\gamma_{j,n}(t_1,...,t_n)=(t_1,...,t_{j-1},\text{max}(t_j,t_{j+1}),t_{j+2},...,t_n).
\end{eqnarray*}
\end{itemize}
The category $\square$ is called the \textbf{cube category with connections} and a presheaf $K: \square^{op} \to \cat{Set}$ is called a \textbf{cubical set with connections}.
Denote by $\square^n \in \cat{cSet}$ the cubical set with connections corepresented by $[1]^n$, i.e.  $$\square^n=\text{Hom}_{\square}( \_, [1]^n).$$
The cube category becomes a monoidal category when equipped with $[1]^n \times [1]^m = [1]^{n+m}$. This induces a monoidal structure $\boxtimes: \cat{cSet}\times \cat{cSet} \to \cat{cSet}$ given explicitly by
$$K \boxtimes L =  \underset{\sigma: \square^n \to K, \tau: \square^m \to L} {\text{colim}} \square^{n+m}.$$
\end{Def}

We now recall the definition of a strong monoidal functor
$$\mathcal{R}: (\cat{Nec}, \vee) \to (\square, \times)$$
constructed in \cite{rivera2018cubical}. 
First, let $\mathcal{R}(\Delta^0)=[1]^0$. On any other necklace $N \in \cat{Nec}$, define $\mathcal{R}(N) = [1]^{\dim(N)}$.
On morphisms, $\mathcal{R}$ is determined by the following rules.
\begin{enumerate}
	\item For any $\partial_j \colon \Delta^n \hookrightarrow \Delta^{n+1}$ such that $0< j<{n+1}$, define $$\mathcal{R}(\partial_j) \colon [1]^{n-1}\to [1]^{n}$$ to be the cubical coface functor  $\delta^0_{j,n-1}$.

	\item For any $\Delta_{[j], [n+1-j]} \colon \Delta^j \vee \Delta^{n+1-j} \hookrightarrow \Delta^{n+1}$ such that $0<j<n+1$, define
	$$\mathcal{R}(\Delta_{[j], [n+1-j]}) \colon [1]^{n-1}\to [1]^{n}$$
	to be the cubical coface functor $\delta^1_{j,n-1}$.

	\item We now consider morphisms of the form $s_j \colon \Delta^{n+1} \twoheadrightarrow \Delta^n$ for $n>0$.
	If $j=0$ or $j=n$, define $$\mathcal{R}(s_j) \colon [1]^n \to [1]^{n-1}$$ to be the cubical codegeneracy functor $\varepsilon_{j,n}$.
	If $0<j<n$, define $$\mathcal{R}(s_j) \colon [1]^n \to [1]^{n-1}$$ to be the cubical coconnection functor $\gamma_{j,n}$.

	\item For $s_0 \colon \Delta^1 \twoheadrightarrow \Delta^0$ define $$\mathcal{R}(s_0) \colon [1]^0 \to [1]^0$$ to be the identity functor.
\end{enumerate}
The functor $\mathcal{R}: \cat{Nec} \to \square$ induces a strong monoidal functor on categories of presheaves $$\mathcal{R}_!: (\cat{nSet}, \vee) \to (\cat{cSet}, \boxtimes)$$ given by 
$$\mathcal{R}_{!}(K) =
\underset{{\mathcal{Y}(N) \to K}} {\text{colim}} \mathcal{Y}_{\square}(\mathcal{R}(N))\cong
\underset{{\mathcal{Y}(N) \to K}}{\text{\colim}} \square^{\dim(N)},
$$
where $\mathcal{Y}_{\square}: \square \to \cat{cSet}$ denotes the Yoneda embedding.

We also have a monoidal functor
$$ \mathcal{T} : (\cat{cSet}, \boxtimes) \to (\cat{sSet}, \times)$$ given by 
$$\mathcal{T}(K) = \underset{\square^n \to K}{\text{colim}} (\Delta^1)^{\times n}.$$
Denote by $$\mathfrak{R}: \cat{Cat}_{\cat{nSet}} \to \cat{Cat}_{\cat{cSet}}$$
and
$$\mathfrak{T}: \cat{Cat}_{\cat{cSet}} \to \cat{Cat}_{\cat{sSet}}$$ the functors obtained by applying $\mathcal{R}_!$ and $\mathcal{T}$, respectively, at the level of mapping spaces. 

We now establish a relationship between $\Pi$ and $\mathfrak{C}$.

\begin{Prop} \label{factorC}
There is a natural isomorphism of functors $$ \mathfrak{T} \circ \mathfrak{R} \circ \Pi \cong \mathfrak{C}.$$
\end{Prop}
\begin{proof}
For any $X \in \cat{sSet}$, the objects of $\mathfrak{T} \circ \mathfrak{R} \circ \Pi(X)$ are, by definition, the elements of $X_0$. For any two $x,y \in X_0$, we have the following natural isomorphisms of simplicial sets 
\begin{align*}\mathfrak{T} (\mathfrak{R}(\Pi(X)))(x,y) \cong \underset{f: N \to X \in (\cat{Nec} \downarrow X)_{x,y}}{\text{colim}} \mathcal{T}(\mathcal{R}_{!}(\mathcal{Y}(N))) \cong
\\
\underset{f: N \to X \in (\cat{Nec} \downarrow X)_{x,y}}{\text{colim}} \mathcal{T}(\square^{\text{dim}(N)})\cong \underset{f: N \to X \in (\cat{Nec} \downarrow X)_{x,y}}{\text{colim}} (\Delta^1)^{\times\text{dim}(N)}.
\end{align*}
Composition is induced by concatenation of necklaces. This is precisely the description of $\mathfrak{C}(X)$ provided by Proposition \ref{dsnecklaces}.
\end{proof} 

 \begin{Rem}

 The composition of functors $\mathfrak{R} \circ \Pi: \cat{sSet} \to \cat{Cat}_{\cat{cSet}}$ was studied in \cite{rivera2018cubical}, where it was denoted by $\mathfrak{C}_{\square_c}$ and called the \textbf{cubical rigidification functor}. Proposition \ref{factorC} above is exactly Proposition 5.3 of \cite{rivera2018cubical}. In this article, we would like to emphasize the factorization of $\mathfrak{C}_{\square_c}$ via $\Pi$. In fact, the framework of necklaces, necklical sets, and necklical categories uses the minimal amount of data necessary to define such a construction. However, the use of cubes is conceptually convenient. 
 \end{Rem}

We recall the definition of the strong monoidal functor of normalized cubical chains 
$$N^{\square}_*: ( \cat{cSet}, \boxtimes) \to (\cat{Ch}_R^{\geq 0}, \otimes).$$
First define $N^{\square}_k( \square^n )$ to be the quotient of the free $R$-module $R[\text{Hom}_{\square}([1]^k, [1]^n)]$ by the sub-$R$-module generated by $D_{k,n}^{\text{deg}} \cup D_{k,n}^{\text{con}}$, where
$$D_{k,n}^{\text{deg}}=\{ \alpha \in \text{Hom}_{\square}([1]^k, [1]^n)] | \alpha = \varepsilon_{j,n} \circ \alpha', \text{ for some } \alpha' \in \text{Hom}_{\square}([1]^k, [1]^{n+1}), 1\leq j\leq n+1\}$$
and
$$D_{k,n}^{\text{con}}=\{ \alpha \in \text{Hom}_{\square}([1]^k, [1]^n)] | \alpha= \gamma_{j,n} \circ \alpha', \text{ for some } \alpha' \in \text{Hom}_{\square}([1]^k, [1]^{n+1}), 1\leq j\leq n\}.$$
The differential $\delta: N^{\square}_k( \square^n ) \to N^{\square}_{k-1}( \square^n )$ is induced by the formula
$$\delta(\alpha)= \sum_{j=1}^k (-1)^j (\alpha \circ \delta^1_{j, k-1} - \alpha \circ \delta^0_{j, k-1}).$$
 Finally, for an arbitrary $K \in \cat{cSet}$, define $$N^{\square}_*(K) = \underset{\square^n \to K}{\text{colim }} N^{\square}_*(\square^n).$$ There is an induced functor $$\mathcal{N}^{\square}_*: \cat{Cat}_{\cat{cSet}} \to \cat{dgCat}_R^{\geq0}$$ given by applying the strong monoidal functor $N^{\square}_*$ at the level of mapping spaces. Denote  $$\mathcal{N}^{nec}_*=\mathcal{N}^{\square}_* \circ \mathfrak{R}: \cat{Cat}_{\cat{nSet}} \to \cat{dgCat}_R^{\geq 0}.$$

\begin{Prop}\label{cobarandnecklaces} There are natural isomorphisms of functors $$\mathcal{N}_*^{nec} \circ \Pi \cong \Lambda \cong \Omega \circ \widetilde{C}^{\Delta}_* ,$$
where $\Lambda: \cat{sSet} \to \cat{dgCat}_R^{\geq0}$ denotes the left adjoint of the differential graded nerve functor $\mathfrak{N}_{dg}:\cat{dgCat}_R^{\geq0} \to \cat{sSet}$ defined in 1.3.1.6 of \cite{higheralgebra}.
\end{Prop}

\begin{proof}
The isomorphism $\mathcal{N}_*^{nec} \circ \Pi \cong \Lambda$ is exactly the statement of Theorem 6.1 in \cite{rivera2018cubical}. The isomorphism $\Lambda \cong \Omega \circ \widetilde{C}^{\Delta}_*$ is shown in Theorem 4.16 of \cite{holstein-lazarev}.
\end{proof}



\subsection{$B_{\infty}$-coalgebras} We introduce the notion of a $B_{\infty}$-coalgebra, i.e. a categorical coalgebra equipped with the extra structure of compatible coassociative coproducts on the chain complexes of morphisms of its cobar construction. This notion is the many object version of the linear dual of a $B_{\infty}$-algebra, as introduced in \cite{GJ94}.
\begin{Def}
A $\mathbf{B}_{\infty}$\textbf{-coalgebra} is a categorical coalgebra $C$ equipped with degree $0$ coassociative coproducts
$$\nabla_{x,y}: \Omega(C)(x,y) \to \Omega(C)(x,y) \otimes \Omega(C)(x,y)$$
for all $x,y \in P_C$ making $\Omega(C)$ into a category enriched over $(\cat{dgCoalg}_R^{\geq 0}, \otimes).$ If $C$ is a differential (non-negatively) graded connected coalgebra, considered as a categorical coalgebra with a single object, then a $B_{\infty}$-coalgebra structure is equivalent to a coassociative coproduct on the ordinary cobar construction $\nabla: \Omega(C) \to \Omega(C) \otimes \Omega(C)$ making the dg algebra $\Omega(C)$ into a dg bialgebra. 
\end{Def}
$B_{\infty}$-coalgebras form a category $\cat{B_{\infty}Coalg}_R$ with morphisms given by morphisms of categorical coalgebras that respect the additional structure. If $C$ is a categorical coalgebra with a single object, i.e. $P_C$ is a singleton, then a $B_{\infty}$-coalgebra structure is equivalent to a coproduct on the dg algebra $\Omega(C)$ making it into a dg bialgebra. The cobar construction may now be considered as a functor
$$\mathbf{\Omega}: \cat{B_{\infty}Coalg}_R \to \cat{Cat}_{\cat{dgCoalg}_{R}^{\geq 0}},$$ where $\cat{Cat}_{\cat{dgCoalg}_{R}^{\geq 0}}$ denotes the category of categories enriched over the monoidal category $(\cat{dgCoalg}_R^{\geq 0}, \otimes).$ A morphism in $\cat{Cat}_{\cat{dgCoalg}_{R}^{\geq 0}}$ will be called a quasi-equivalence if it is a quasi-equivalence of underlying dg categories.

\subsection{Normalized chains as a $B_{\infty}$-coalgebra}

 Let $\widetilde{C}^{\Delta}_*: \cat{sSet} \to \cat{cCoalg}_R$ be the normalized chains functor as defined in Definition \ref{chains}. For any $X \in \cat{sSet}$ and $x,y \in X_0$ we construct a natural coassociative coproduct
$$\nabla_{x,y}: \Omega \widetilde{C}^{\Delta}_*(X)(x,y) \to  \Omega \widetilde{C}^{\Delta}_*(X)(x,y)  \otimes  \Omega \widetilde{C}^{\Delta}_*(X)(x,y).$$ 

Note that we have a canonical isomorphism of dg $R$-modules $$(N^{\square}_*(\square^n), \delta) \cong (N^{\Delta}_*(\Delta^1), \partial)^{\otimes n}.$$
Since $N^{\Delta}_*(\Delta^1)$ is a dg coalgebra when equipped with Alexander-Whitney coproduct, the above isomorphism yields a dg coalgebra structure on $N^{\square}_*(\square^n)$ via the monoidal structure of $\cat{dgCoalg}_R^{\geq 0}$. This gives rise to a coassociative coproduct of degree $0$ denoted by  $$\nabla: N^{\square}_*( \square^n ) \to N^{\square}_*( \square^n ) \otimes N^{\square}_*( \square^n )$$ and known as the \textbf{Serre coproduct}. We write $C^{\square}_*(\square^n)= (N^{\square}_*(\square^n), \delta, \nabla) \in \cat{dgCoalg}_R^{\geq0}.$ Finally, for an arbitrary cubical set with connections $K$, define $$C^{\square}_*(K) = \underset{\square^n \to K}{\text{colim }} C^{\square}_*(\square^n) \in \cat{dgCoalg}_R^{\geq 0}.$$  
Note that the underlying dg $R$-module of $C^{\square}_*(K)$ is $(N^{\square}_*(K), \delta)$ so we write $$C^{\square}_*(K)= (N^{\square}_*(K), \delta, \nabla) \in \cat{dgCoalg}_R^{\geq 0}.$$ This construction gives rise to a strong monoidal functor
$$C^{\square}_*: (\cat{cSet}, \boxtimes) \to (\cat{dgCoalg}_R^{\geq 0}, \otimes).$$
Consequently, we obtain an induced functor
$$\mathcal{C}_*^{\square}: \cat{Cat}_{\cat{cSet}} \to \cat{Cat}_{\cat{dgCoalg}_R^{\geq 0}}$$
by applying $C^{\square}_*$ at the level of mapping spaces. Define
$$\mathcal{C}^{nec}_*: \cat{Cat}_{\cat{nSet}} \to \cat{Cat}_{\cat{dgCoalg}_R^{\geq 0}}$$ to be
 the composition 
$$ \cat{Cat}_{\cat{nSet}}  \xrightarrow{\mathfrak{R}} \cat{Cat}_{\cat{cSet}} \xrightarrow{\mathcal{C}^{\square}_*} \cat{Cat}_{\cat{dgCoalg}_R^{\geq 0}}.$$ The following is now a straightforward consequence of the definitions.
\begin{Prop}\label{coalgebraoncobar}
The composition of functors $\mathcal{C}^{nec}_* \circ \Pi$ provides a lift of  $$\mathcal{N}^{nec}_* \circ \Pi : \cat{sSet} \to \cat{dgCat}_R^{\geq 0}$$ to the category $\cat{Cat}_{\cat{dgCoalg}_R^{\geq 0}}$ of categories enriched over dg coalgebras. 
\end{Prop}
By Proposition \ref{cobarandnecklaces}, we have a natural  isomorphisms of functors $$\mathcal{N}^{nec}_* \circ \Pi \cong \Omega \circ \widetilde{C}_*^{\Delta}.$$ Hence, Proposition \ref{coalgebraoncobar} provides a natural $B_{\infty}$-coalgebra structure on $\widetilde{C}_*^{\Delta}(X)$ for any simplicial set $X$. This gives rise to a functor
$$\widetilde{\mathbf{C}}_*: \cat{sSet} \to \cat{B_{\infty}Coalg}_R.$$

The following commutative diagram summarizes the main constructions of this section
\begin{equation}\label{diagram1}
    \begin{tikzcd}
    {\cat{Quiv}} \arrow[d, "j"] \arrow[r, "F"]  & {\cat{Cat}} \arrow[d, "\tau"]
   \\
{\cat{sSet}} \arrow[d, "\widetilde{\mathbf{C}}_*"] \arrow[r, "\Pi"]  & {\cat{Cat}_{\cat{nSet}}} \arrow[d, "\mathcal{C}^{nec}_*"] 
\\
{\cat{B_{\infty}Coalg}_R} \arrow[r, "\mathbf{\Omega}"]  & {\cat{Cat}_{\cat{dgCoalg}_{R}^{\geq 0}}}.
    \end{tikzcd}
\end{equation} 
In the above diagram, $\tau$ is the functor that associates to any category the induced necklical category with discrete mapping spaces, $j$ sends any quiver to its associated $1$-skeletal simplicial set, and $F$ sends any quiver to the free category that it generates. Consequently, the functor $\Pi: \cat{sSet} \to \cat{Cat}_{\cat{nSet}}$ may be regarded as a generalization of the free category generated by a quiver but now taking into account higher structure. 

\subsection{The extended cobar construction} We may use the coalgebra structure on the morphisms of the cobar construction of a $B_{\infty}$-coalgebra to define an extended version of $\mathbf{\Omega}$ as follows. Let $$S: (\cat{dgCoalg}_R^{\geq 0}, \otimes) \to (\cat{Set}, \times)$$ be the strong monoidal functor defined by $$S(C,\partial, \Delta) = \{ x \in C | \Delta(x) = x \otimes x\ \text{and } \varepsilon(x)=1\},$$ where $\varepsilon: C \to R$ denotes the counit. The elements of $S(C,\partial, \Delta)$ are called the \textbf{set-like elements}. Note that all set-like elements in a dg coalgebra $C \in \cat{dgCoalg}_R^{\geq 0}$ are of degree $0$. There is an induced functor
$$\mathcal{S}: \cat{Cat}_{\cat{dgCoalg}_R^{\geq 0}} \to \cat{Cat}$$ given by applying $S$ on mapping spaces. Given any $\cat{C} \in \cat{Cat}_{\cat{dgCoalg}_R^{\geq 0}}$ define $$\cat{C}[\mathcal{S}(\cat{C})^{-1}] \in \cat{Cat}_{\cat{dgCoalg}_R^{\geq 0}}$$
by formally (strictly) inverting all set-like elements in all the dg coalgebras of morphisms in $\cat{C}$ and declaring the newly added inverses to be set-like. 
\begin{Rem} We will only apply the above construction when $\cat{C}$ is a cofibrant dg category, in which case the strict localization is a homotopy invariant with respect to quasi-equivalences. 
\end{Rem}
\begin{Def}
The \textbf{extended cobar construction} is the functor
$$\widehat{\mathbf{\Omega}}: \cat{B_{\infty}Coalg}_R \to \cat{Cat}_{\cat{dgCoalg}_R^{\geq 0}}$$
given on any $C \in \cat{B_{\infty}Coalg}_R$ by  
$$\widehat{\mathbf{\Omega}}(C)= \mathbf{\Omega}(C)[ \mathcal{S}(\mathbf{\Omega}(C))^{-1}].$$
\end{Def}
We now have the following localized version of diagram \ref{diagram1}. 
\begin{equation}\label{diagram1}
    \begin{tikzcd}
    {\cat{Quiv}} \arrow[d, "j"] \arrow[r, "i \circ L \circ F"]  & {\cat{Cat}} \arrow[d, "\tau"]
   \\
{\cat{sSet}} \arrow[d, "\widetilde{\mathbf{C}}_*"] \arrow[r, "\widehat{\Pi}"]  & {\cat{Cat}_{\cat{nSet}}} \arrow[d, "\mathcal{C}^{nec}_*"] 
\\
{\cat{B_{\infty}Coalg}_R} \arrow[r, "\widehat{\mathbf{\Omega}}"]  & {\cat{Cat}_{\cat{dgCoalg}_{R}^{\geq 0}}}.
    \end{tikzcd}
\end{equation} 
Here $\widehat{\Pi}(X)$ is defined by formally inverting all $0$-dimensional necklaces in all the mapping spaces of $\Pi(X).$ The functors $L$ and $i$ are the localization and inclusion functors defined in Definition \ref{localizationofcategories}.

\subsection{Proof of Theorem \ref{theorem2}}
Part $(1)$ of Theorem \ref{theorem2} follows immediately from the above construction. Part $(2)$ follows since , for any simplicial set $X$ and vertices $x,y \in X_0$, we have a natural isomorphism of dg coalgebras 
$$\mathbf{\Omega}( \widetilde{\mathbf{C}}_*(X))(x,y) \cong C^{\square}_*( \mathcal{R}_! ( \Pi (X)(x,y)) ).$$
Note that the set-like elements of the dg coalgebra $C^{\square}_*(K)$ for any $K \in \cat{cSet}$ are exactly the $0$-cubes in $K$. The $0$-cubes in $\mathcal{R}_!(\Pi(X)(x,y)) \in \cat{cSet}$ are all necklaces of $1$-simplices in $X$ connecting $x$ and $y$; these are precisely the elements of the set $F (\mathcal{Q}(X))(x,y) $ of morphisms from $x$ to $y$ in the free category generated by the underlying quiver of $X$. 

We prove part $(3)$. It follows from Proposition \ref{L1andL} and Corollary \ref{corthm1} that the maps 
$$\widehat{\mathfrak{C}}(X) \xrightarrow{\mu_X} \iota \mathcal{L} \mathfrak{C}(X) \xrightarrow{\iota(\mathcal{L}(Sz_X))} \mathbb{G}(X)$$
are weak equivalences of simplicial categories. Since the simplicial chains functor sends weak homotopy equivalences to quasi-isomorphisms, the induced map
$$\mathcal{C}^{\Delta}_*(\iota(\mathcal{L}(Sz_X)) \circ \mu_X): \mathcal{C}^{\Delta}_*(\widehat{\mathfrak{C}}(X)) \to \mathcal{C}^{\Delta}_*(\mathbb{G}(X))$$
is a quasi-equivalence in $\cat{Cat}_{\cat{dgCoalg}_R^{\geq 0}}.$

We now construct a natural quasi-equivalence
$$\mathbb{T}_X: \widehat{\mathbf{\Omega}}(\widetilde{\mathbf{C}}_*(X)) \to \mathcal{C}^{\Delta}_*(\widehat{\mathfrak{C}}(X)).$$ First recall that for any $K \in \cat{cSet}$ we have a natural quasi-isomorphism of dg coalgebras 
$$EZ^{\square}_K: C^{\square}_*(K) \to C^{\Delta}_*(\mathcal{T}(K))$$
where $\mathcal{T}: \cat{cSet} \to \cat{sSet}$ is the triangulation functor. The map $EZ^{\square}_K$ is induced by the classical Eilenberg-Zilber shuffle map of dg coalgebras
$$EZ: C^{\square}_*(\square^{n})\cong C^{\Delta}_*(\Delta^1)^{\otimes n} \to C^{\Delta}_*((\Delta^{1})^{\times n})$$
after using the identifications
$C_*^{\square}(K)\cong \underset{\square^n \to K}{\text{colim}} C_*^{\square}(\square^n)$ and $C^{\Delta}_*(\mathcal{T}(K)) \cong \underset{\square^n \to K}{\text{colim}} C_*^{\Delta}((\Delta^1)^{\times n})$. The fact that $EZ^{\square}_K$ is a quasi-isomorphism follows from a standard acyclic models argument. The fact that $EZ^{\square}_K$ is a map of dg coalgebras follows from 17.6 in \cite{EilMoo66}, which says that the classical Eilenberg-Zilber shuffle map preserves coalgebra structures. Note that $EZ^{\square}_K$ sends an $n$-cube in $C_n^{\square}(K)$ to a signed sum of $n!$ $n$-simplices in $C_*^{\Delta}(\mathcal{T}(K))$ each labeled by a permutation of the set $\{1,2,...,n\}$.  

By applying this construction at the level of morphisms, we obtain an induced natural quasi-equivalence
$$\mathcal{EZ}^{\square}_{\cat{C}}: \mathcal{C}^{\square}_*(\cat{C}) \to \mathcal{C}^{\Delta}_* (\mathfrak{T} (\cat{C}))$$ for any $\cat{C} \in \cat{Cat}_{\cat{cSet}}.$

Applying the above discussion to $\cat{C}= \mathfrak{R} (\Pi ( X))$ together with Proposition \ref{factorC} and the commutativity of the bottom square of \ref{diagram1}, we obtain a natural quasi-equivalence
\begin{equation}\label{EZcategorical}
\mathbf{\Omega}(\widetilde{\mathbf{C}}_*(X))\cong \mathcal{C}^{nec}_* (\Pi(X)) \xrightarrow{\mathcal{EZ}^{\square}_{\mathfrak{R} (\Pi ( X))}} \mathcal{C}^{\Delta}_*(\mathfrak{T} (\mathfrak{R} (\Pi ( X))))\cong \mathcal{C}^{\Delta}_*(\mathfrak{C}(X)).
\end{equation}

Finally, the quasi-equivalence \ref{EZcategorical} induces a quasi-equivalence $\mathbb{T}_X: \widehat{\mathbf{\Omega}}(\widetilde{\mathbf{C}}_*(X)) \to \mathcal{C}^{\Delta}_*(\widehat{\mathfrak{C}}(X))$
after localizing at the set-like elements. This follows since
$$\widehat{\mathbf{\Omega}}(\widetilde{\mathbf{C}}_*(X))= \mathbf{\Omega}(\widetilde{\mathbf{C}}_*(X))[ \mathcal{S}(\mathbf{\Omega}(\widetilde{\mathbf{C}}_*(X)))^{-1}],$$

$$\mathcal{C}^{\Delta}_*(\widehat{\mathfrak{C}}(X))\cong \mathcal{C}^{\Delta}_*(\mathfrak{C}(X))[\mathcal{S}(\mathcal{C}^{\Delta}_*(\mathfrak{C}(X))^{-1}],$$ and
$\mathcal{EZ}^{\square}_{\mathfrak{R} (\Pi ( X))}$ induces an isomorphism of categories

$$\mathcal{S}(\mathbf{\Omega}(\widetilde{\mathbf{C}}_*(X))) \xrightarrow{\cong} \mathcal{S}(\mathcal{C}^{\Delta}_*(\mathfrak{C}(X)).$$
\subsection{The extended cobar construction as a model for the path category}
We now state some consequences of Theorems \ref{theorem1} and \ref{theorem2}. The first is that the extended cobar construction $$\widehat{\mathbf{\Omega}}: \cat{B_{\infty}Coalg}_R \to \cat{Cat}_{\cat{dgCoalg}_R^{\geq 0}}$$ yields a model for the path category when applied to the $B_{\infty}$-coalgebra of normalized chains. This is Theorem \ref{algmodelforpathcat}, which we restate in the following.
\begin{Cor}
For any simplicial set $X \in \cat{sSet}$, the dg coalgebra enriched categories $\widehat{\mathbf{\Omega}}(\widetilde{\mathbf{C}}_*(X))$ and $\mathcal{C}_*^{\Delta}(\mathbb{P}(X))$ are naturally quasi-equivalent.
\end{Cor}
\begin{proof}
As we have seen in the previous section, we have quasi-equivalence
$$\widehat{\mathbf{\Omega}}(\widetilde{\mathbf{C}}_*(X)) \xrightarrow{\simeq} \mathcal{C}^{\Delta}_*(\mathbb{G}(X))$$
in $\cat{Cat}_{\cat{dgCoalg}_R^{\geq 0}}$. By Theorem \ref{equivalencesofsimplicialcategories}, the simplicial categories $\mathbb{G}(X)$ and $\mathbb{P}(X)$ are naturally weakly equivalent. This implies that $\mathcal{C}^{\Delta}_*(\mathbb{G}(X))$ and $\mathcal{C}^{\Delta}_*(\mathbb{P}(X))$ are naturally quasi-equivalent as objects in $\cat{Cat}_{\cat{dgCoalg}_R^{\geq 0}}$.
\end{proof}
When $X$ is a reduced simplicial set then $\mathbf{\Omega}(\widetilde{\mathbf{C}}_*(X))$ has a single object, so it may be considered as a dg bialgebra. In fact, there is an isomorphism of underlying dg algebras
$$\varphi: \Omega(C^{\Delta}_*(X)) \xrightarrow{\cong} \mathbf{\Omega}(\widetilde{\mathbf{C}}_*(X)),$$
where the left hand side is now the classical cobar construction of the connected dg coalgebra of normalized chains on $X$. The map $\varphi$ is determined by defining $\varphi(s^{-1}\sigma)= s^{-1}\sigma - 1_R$ if $\sigma \in X_1$ and $\varphi(s^{-1}\sigma)= s^{-1}\sigma$ if $\sigma \in X_n$ for $n>1.$ Then $\varphi$ is extended to monomials of arbitrary length in $\Omega(C^{\Delta}_*(X))$ as an algebra map. This induces a natural isomorphism of dg algebras after localizing
$$\varphi: \widehat{\Omega}(C^{\Delta}_*(X)) \xrightarrow{\cong} \widehat{\mathbf{\Omega}}(\widetilde{\mathbf{C}}_*(X)),$$
where the left hand side is now the extended cobar construction as defined in section 1.2 in \cite{hess2010loop}. Since  $\widehat{\mathbf{\Omega}}(\widetilde{\mathbf{C}}_*(X))$ has a dg bialgebra structure when equipped with the Serre coproduct, one obtains a natural dg bialgebra structure on $\widehat{\Omega}(C^{\Delta}_*(X))$ via the isomorphism $\varphi.$
The dg coalgebra structure on $\Omega(C^{\Delta}_*(X))$ is induced by the homotopy Gerstenhaber coalgebra structure on $C^{\Delta}_*(X)$, as explained in Appendix A of \cite{franz2021szczarba}. This coproduct on the cobar construction was also studied by Baues in \cite{Bau81}.

Hess and Tonks also define a quasi-isomorphism of dg algebras
$$\phi: \widehat{\Omega}(C^{\Delta}_*(X)) \to C^{\Delta}_*(G^{Kan}(X)),$$
in terms of the Szczarba operators, see Theorem 7 of \cite{hess2010loop}. The map $\phi$ is induced by Szczarba's twisting cochain $t: C^{\Delta}_*(X) \to C^{\Delta}_*(G^{Kan}(X))$ as recalled in 5.3 of \cite{franz2021szczarba}. As a consequence of Theorem \ref{theorem2} we obtain that Hess and Tonk's map $\phi$ is comultiplicative, which is the main theorem in \cite{franz2021szczarba}. 

\begin{Cor} For any $0$-reduced simplicial set $X$, the composition  $$\widehat{\Omega}(C^{\Delta}_*(X)) \xrightarrow{\varphi} \widehat{\mathbf{\Omega}}(\widetilde{\mathbf{C}}_*(X)))  \xrightarrow{\mathbb{T}_X} C_*(\widehat{\mathfrak{C}}(X)) \xrightarrow{C_*(\iota (\mathcal{L} (Sz_X)) \circ \mu_X)}C^{\Delta}_*(G^{Kan}(X))$$ is a quasi-isomorphism of dg bialgebras. Furthermore, we have
$$C_*(\iota (\mathcal{L} (Sz_X)) \circ \mu_X) \circ \mathbb{T}_X \circ \varphi = \phi.$$
\end{Cor} 
\begin{proof}
It follows from Theorem \ref{theorem2} that the composition of maps is a quasi-isomorphism of dg bialgebras. The fact that this composition coincides with $\phi$ follows by unraveling the formula explicitly. More precisely, on an element $\sigma \in X_{n+1}$, considered as a necklace of length $1$ and dimension $n$, $(C_*(\iota (\mathcal{L} (Sz_X)) \circ \mu_X) \circ \mathbb{T}_X \circ \varphi)(\sigma)$ coincides with Szczarba's twisting cochain applied to $\sigma$: it is given by applying the Szczarba map to the signed sum of $n!$ $n$-simplices each labeled by the permutation corresponding to an $n$-simplex in $(\Delta^1)^{\times n}$ using $\varphi(\sigma)=\sigma -1_R$ if $n=0.$
\end{proof}
\appendix
\section{Model Structures}\label{appendixmodel}

In this appendix we collect the various model structures used throughout this paper. If $\cat{C}$ is a model category, we will denote the underlying hom set between two objects $X,Y \in \cat{C}$ by $\cat{C}(X,Y)$, and the set of homotopy classes of maps by $\cat{C}[X,Y]$.

\subsection{The Kan-Quillen model structure on simplicial sets}
A simplicial set $K$ is a \textit{Kan complex} if any horn $f: \Lambda^n_k \to K$ can be extended to $\tilde{f}: \Delta^n \to K$ for $0\leq k \leq n$. Given any Kan complex $K$ denote by $\cat{sSet}_{KQ}[X,K]$ the set of simplicial homotopy classes of maps from $X$ to $K$. 

\begin{Thm} There exists a proper and cofibrantly generated model structure on the category of simplicial sets in which
\begin{enumerate}
    \item a map $f: X \to Y$ is a weak equivalence if for every Kan complex $K$ the induced map of sets of simplicial homotopy classes of maps $\cat{sSet}_{KQ}[Y,K] \to \cat{sSet}_{KQ}[X,K]$ is an isomorphism, 
    \item the cofibrations are the monomorphisms, and
    \item the fibration are the Kan fibrations i.e. maps $f: X \to Y$ with the right lifting property with respect to horn inclusions $\Lambda^{n}_k \to \Delta^n$ for all $n \geq 1$ and $0 \leq k \leq n.$
\end{enumerate}
\end{Thm}
We call the above model category structure the \textit{Kan-Quillen model structure} on simplicial sets. We denote it by $\cat{sSet}_{KQ}$. All objects in $\cat{sSet}_{KQ}$ are cofibrant. The fibrant objects in $\cat{sSet}_{KQ}$ are precisely the Kan complexes.  We will call the weak equivalences in $\cat{sSet}_{KQ}$ simply  \textit{weak homotopy equivalences of simplicial sets}. 

The Kan-Quillen model category is a model for the homotopy theory of homotopy types.

\begin{Thm} The adjunction $$| - |: \cat{sSet}_{KQ} \rightleftarrows \cat{Top}_Q: \text{Sing}$$
defines a Quillen equivalence of model categories, where $\cat{Top}_Q$ denotes the classical model structure on topological spaces.
\end{Thm}

\subsection{The Joyal model structure on simplicial sets}

A simplicial set $Q$ is a \textit{quasi-category} if any horn $f: \Lambda^n_k \to Q$ can be extended to $\tilde{f}: \Delta^n \to Q$ for $0<k<n$.

Consider the groupoid $\mathbb{J}^n = L[n]$, where $[n]$ is the category determined by the poset $\{ 0 \to 1 \to ... \to n \}$ and $L$ is the localization functor from categories to groupoids. We will abuse notation and write $\mathbb{J} = \mathbb{J}^1$. Let $N(\mathbb{J}^n)$ be the Kan complex obtained by applying the nerve functor $N$. Given a quasi-category $Q$ and a simplicial set $X$, two maps $f, g: X \to Q$ are said to be $\mathbb{J}$-\textit{homotopic} if there is a map $H: X \times N(\mathbb{J}) \to Q$ such that $H \circ i_0 =f$ and $H \circ i_1=g$, where $i_0, i_1: X \to X \times N(\mathbb{J})$ are the natural inclusions. The notion of $\mathbb{J}$-homotopy defines an equivalence relation on the set $\cat{sSet}(X,Q)$. Denote the set of equivalence classes by $\cat{sSet}_J[X,Q]$. 

 \begin{Thm} There exists a left proper and cofibrantly generated model structure on the category of simplicial sets in which
 \begin{enumerate}
     \item the weak equivalences are maps $f: X \to Y$ such that for every quasi-category $Q$, the induced map $\cat{sSet}_J[Y,Q] \to \cat{sSet}_J[X,Q]$ is an isomorphism of sets,
     \item the cofibrations are the monomorphisms, and
     \item the fibrant objects are the quasi-categories.
 \end{enumerate}
\end{Thm}

We call the above model category structure the \textit{Joyal model structure} on simplicial sets. We denote it by $\cat{sSet}_{J}$. All objects in $\cat{sSet}_{J}$ are cofibrant.  We will call a weak equivalence in $\cat{sSet}_{J}$ a \textit{Joyal equivalence}. 

For any quasi-category $Q$ and simplicial set $X$, the \textit{derived mapping space} $\mathbf{R}\cat{sSet}_J(X,Q)$ may be defined as the simplicial set whose $n$-simplices are given by maps $X \times N(\mathbb{J}^n) \to Q$. 

A fundamental fact in the theory of quasi-categories, used repeatedly in this article, is the following.

\begin{Thm} [{\cite[Corollary 1.4]{joyal02}}]  \label{Kaniff} A quasi-category $Q$ is a Kan complex if and only if its homotopy category $\pi_0(\mathfrak{C}(Q))$ is a grupoid.
\end{Thm}

We also record that $\cat{sSet}_{KQ}$ is a left Bousfield localization of $\cat{sSet}_J$. 

\begin{Prop}[{\cite[Propositions 5.9, 5.10]{riehl2008model}}\label{Bousfieldloc}] The Kan-Quillen model category $\cat{sSet}_{KQ}$ is a left Bousfield localization of the Joyal model category $\cat{sSet}_{J}$. In particular, a Joyal equivalence is a weak homotopy equivalence and every weak homotopy equivalence between Kan complexes is a Joyal equivalence. 
\end{Prop}

\begin{Rem}\label{locatp} In fact, the model category $\cat{sSet}_{KQ}$ is the left Bousfield localization of $\cat{sSet}_J$ at the single morphism $p: \Delta^1 \to \Delta^0$. For any quasi-category $Q$, the map $p^*: \mathbf{R}\cat{sSet}_J(\Delta^0,Q) \to \mathbf{R}\cat{sSet}_J(\Delta^1,Q)$  between derived mapping spaces is a weak homotopy equivalence if and only if the homotopy category of $Q$ is a groupoid. This implies that the fibrant objects in the left Bousfield localization are precisely quasi-categories and the weak equivalences are the Joyal equivalences.
\end{Rem}

\subsection{The Dwyer-Kan model structure on simplicial groupoids}

\begin{Thm} [{\cite{dwyer_kan_1984}, \cite[Theorem 2.2]{bergner2008adding}}] \label{sgpdDK}
 There exists a cofibrantly generated model category structure on the category of simplicial groupoids such that
 \begin{enumerate}
    \item the weak equivalences are maps $F: C \to D$ such that
    \begin{enumerate}
        \item the induced functor $$\pi_0 C \xrightarrow{\pi_0 F} \pi_0 D$$ is an equivalence of categories, and
        \item for all objects $x,y \in C$, the induced morphism $$C(x,y) \xrightarrow{F} D(Fx, Fy)$$
        is a weak homotopy equivalence of simplicial sets,
    \end{enumerate}
    \item the fibrations are maps $f: C \to D$ such that
    \begin{enumerate}
        \item the induced functor $$\pi_0 C \xrightarrow{\pi_0 F} \pi_0 D$$ is an isofibration of categories, and 
        \item for all objects $x,y \in C$, the induced morphism $$C(x,y) \xrightarrow{F} D(Fx,Fy)$$ is a Kan fibration.
    \end{enumerate}
\end{enumerate}
\end{Thm}

We call the above model category structure the \textit{Dwyer-Kan model structure} on simplicial groupoids. We denote it by $\cat{Gpd}_{\cat{sSet}, DK}$. Every object in $\cat{Gpd}_{\cat{sSet}, DK}$ is fibrant. This model category structure models the homotopy theory of homotopy types.

\begin{Thm}[{\cite[Theorem 3.3]{dwyer_kan_1984}}] \label{kan Qequivalence}
The adjunction \begin{equation}
G^{Kan}: \cat{sSet}_{KQ} \rightleftarrows \cat{Gpd}_{\cat{sSet},DK}: \conj{W}^{Kan}
\end{equation}
defines a Quillen equivalence. 
\end{Thm}

The above result is a generalization of a "one object" version originally due to Kan. More precisely, there are model structures on the category $\cat{sSet}^0$ and on the category of simplicial groups $\cat{sGrp}$ such that the adjunction $G^{Kan}: \cat{sSet}^0 \rightleftarrows \cat{sGrp}: \conj{W}^{Kan}$ becomes a Quillen equivalence. A complete proof of this result may be found in \cite{goerss2009simplicial}. We denote these model structures by $\cat{sSet}^0_{KQ}$ and $\cat{sGrp}_{KQ}.$

\subsection{The Bergner model structure on simplicial categories}

\begin{Thm}[{\cite{bergner2004model}, \cite[Theorem A.3.2.4]{lurie2006higher}}] \label{scatB} There exists a proper and cofibrantly generated model category structure on the category of simplicial categories such that 

\begin{enumerate}
    \item the weak equivalences are maps $F: C \to D$ such that
    \begin{enumerate}
        \item the induced functor $$\pi_0 C \xrightarrow{\pi_0 F} \pi_0 D$$ is an equivalence of categories
        \item for all objects $x,y \in C$, the induced morphism $$C(x,y) \xrightarrow{F} D(Fx, Fy)$$
        is a weak homotopy equivalence, and
    \end{enumerate}
    \item the fibrations are maps $F:C\to D$ such that
    \begin{enumerate}
        \item the induced functor $$\pi_0 C \xrightarrow{\pi_0 F} \pi_0 D$$
        is an isofibration of categories, and
         \item for all objects $x,y \in C$, the induced morphism
         $$C(x,y) \xrightarrow{F} D(Fx,Fy)$$ is a Kan fibration.
    \end{enumerate}
\end{enumerate}
\end{Thm}

We call the above model category structure the \textit{Bergner model structure} on simplicial categories. We denote it by $\cat{Cat}_{\cat{sSet},B}$. This model category structure models the homotopy theory of infinity categories. The following result has been proven by Joyal and Lurie independently. 

\begin{Thm} \label{rigidification} The adjunction
$$\mathfrak{C}: \cat{sSet}_{J} \rightleftarrows \cat{Cat}_{\cat{sSet},B}: \mathfrak{N}$$ defines a Quillen equivalence. 
\end{Thm}
In particular, a map of simplicial sets $f: X \to Y$ is a Joyal equivalence if and only if $\mathfrak{C}(f): \mathfrak{C}(X) \to \mathfrak{C}(Y)$ is a weak equivalence of simplicial categories. Proposition \ref{Bousfieldloc} implies that if $X$ and $Y$ are Kan complexes, then a map $f: X \to Y$ is a weak homotopy equivalence if and only if $\mathfrak{C}(f): \mathfrak{C}(X) \to \mathfrak{C}(Y)$ is a weak equivalence of simplicial categories. 

\subsection{A model structure on simplicial categories modelling homotopy types}

Although we do not use it explicitly in this article, some of our constructions may be explained in the following model category structure on $\cat{Cat}_{\cat{sSet}}$ modelling homotopy types.

\begin{Thm} The adjunction $\mathfrak{C}: \cat{sSet} \rightleftarrows \cat{Cat}_{\cat{sSet}}: \mathfrak{N}$ induces a Quillen equivalence between $$\mathfrak{C}: \cat{sSet}_{KQ} \rightleftarrows L_S\cat{Cat}_{\cat{sSet},B}: \mathfrak{N}$$ where $L_S\cat{Cat}_{\cat{sSet}, B}$ denotes the left Bousfield localization of the Bergner model structure at the singleton $S=\{ \mathfrak{C}(p): \mathfrak{C}(\Delta^1) \to \mathfrak{C}(\Delta^0)\}$. 
\end{Thm}
\begin{proof}
The left Bousfield localization $L_S\cat{Cat}_{\cat{sSet}, B}$ exists by Theorem 4.7 in \cite{barwick2010left} since $\cat{Cat}_{\cat{sSet}, B}$ is left proper and combinatorial. The result then follows from Proposition \ref{Bousfieldloc} and Remark \ref{locatp} together with the functoriality of left Bousfield localization along Quillen equivalences as proven in Theorem 3.3.20 in \cite{hirschhorn2009model}. 
\end{proof}
\begin{Rem}
The weak equivalences in $L_S\cat{Cat}_{\cat{sSet}, B}$ are maps of simplicial categories that become weak equivalences in the Bergner model structure after applying a cofibrant replacement functor followed by the localization functor. Fibrant objects are Kan enriched simplicial categories whose homotopy category is a groupoid.
\end{Rem}
Since the adjunction $$\iota: \cat{Gpd}_{\cat{sSet},DK} \rightleftarrows L_S\cat{Cat}_{\cat{sSet},B}: \mathcal{L}$$ is clearly a Quillen equivalence, we obtain another localized version of Theorem \ref{rigidification}.
\begin{Cor} \label{Qequi}
The Quillen adjunction
$$\mathcal{L}\mathfrak{C} : \mathsf{sSet}_{KQ} \rightleftarrows \mathsf{Gpd}_{\mathsf{sSet},DK}: \mathfrak{N} \iota$$
defines a Quillen equivalence. 
\end{Cor}

\printbibliography

\end{document}